\documentclass[a4paper,11pt,reqno,noindent]{amsart}
\usepackage[centertags]{amsmath}
\usepackage{mathtools} 
\usepackage{amsfonts,amssymb,amsthm,dsfont,cases,amscd,esint,enumerate}
\usepackage[T1]{fontenc}
\usepackage[english,frenchb]{babel}
\usepackage[applemac]{inputenc}
\usepackage{newlfont,cancel}
\usepackage{color}
\usepackage[body={15cm,21.5cm},centering]{geometry} 
\usepackage{fancyhdr}
\usepackage{enumerate}

\theoremstyle{plain}
\newtheorem{theor10}{Theorem}
\newenvironment{theor1}
  {\pushQED{\qed}\begin{theor10}}
  {\popQED\end{theor10}}
\newtheorem{cor10}{Corollary}
\newenvironment{cor1}
  {\pushQED{\qed}\begin{cor10}}
  {\popQED\end{cor10}}
\newtheorem{lem0}{Lemma}[section]
\newenvironment{lem}
  {\pushQED{\qed}\begin{lem0}}
  {\popQED\end{lem0}}
\theoremstyle{plain}
\newtheorem{theor0}[lem0]{Theorem}
\newenvironment{theor}
  {\pushQED{\qed}\begin{theor0}}
  {\popQED\end{theor0}}
\newtheorem{prop0}[lem0]{Proposition}
\newenvironment{prop}
  {\pushQED{\qed}\begin{prop0}}
  {\popQED\end{prop0}}
\newtheorem{cor0}[lem0]{Corollary}
\newenvironment{cor}
  {\pushQED{\qed}\begin{cor0}}
  {\popQED\end{cor0}}
\theoremstyle{definition}
\newtheorem{rems0}[lem0]{Remarks}
\newenvironment{rems}
  {\pushQED{\qed}\begin{rems0}}
  {\popQED\end{rems0}}
\newtheorem{rem0}[lem0]{Remark}
\newenvironment{rem}
  {\pushQED{\qed}\begin{rem0}}
  {\popQED\end{rem0}}

\numberwithin{equation}{section}

\newcommand{\N}{\mathbb N}
\newcommand{\e}{\varepsilon}

\newcommand{\R}{\mathbb R}
\newcommand{\Z}{\mathbb Z}
\newcommand{\C}{\mathbb C}
\newcommand{\Pm}{\mathbb P^\circ}
\newcommand{\D}{\mathbb D}
\newcommand{\T}{\mathbb T}
\newcommand{\Pc}{\mathcal P}
\newcommand{\G}{\mathcal G}
\newcommand{\Lc}{\mathcal L}
\newcommand{\Tc}{\mathcal T}
\newcommand{\Hc}{\mathcal H}

\newcommand{\Nc}{\mathcal N}
\newcommand{\Dc}{\mathcal D}
\newcommand{\Id}{\operatorname{Id}}
\newcommand{\Var}{\operatorname{Var}^\circ}

\newcommand{\Img}{\operatorname{Im}}
\newcommand{\LB}{{\operatorname{LB}}}
\newcommand{\loc}{{\operatorname{loc}}}
\newcommand{\Sym}{{\operatorname{Sym}}}
\newcommand{\Ld}{\operatorname{L}}
\newcommand{\supp}{\operatorname{supp}}
\newcommand{\step}[1]{\noindent \textit{Step} #1.}
\newcommand{\cov}[2]{\operatorname{Cov}^\circ\!\left[{#1};{#2}\right]}
\newcommand{\covr}[2]{\operatorname{Cov}\!\left[{#1};{#2}\right]}
\newcommand{\covm}[2]{\operatorname{Cov}^\circ\!\big[{#1};{#2}\big]}
\newcommand{\var}[1]{\operatorname{Var}^\circ\!\left[{#1}\right]}
\newcommand{\varr}[1]{\operatorname{Var}\!\left[{#1}\right]}

\newcommand{\E}{\mathbb{E}^\circ}
\newcommand{\expec}[1]{\mathbb{E}^\circ\!\left[{#1}\right]}
\newcommand{\expecm}[1]{\mathbb{E}^\circ\!\big[{#1}\big]}
\newcommand{\varm}[1]{\operatorname{Var}^\circ\!\!\big[{#1}\big]}
\newcommand{\varM}[1]{\operatorname{Var}^\circ\!\!\bigg[{#1}\bigg]}
\newcommand{\expecM}[1]{\mathbb{E}^\circ\!\bigg[{#1}\bigg]}

\newcommand{\dW}[2]{\operatorname{d}_{\operatorname{W}}\left({#1}\,;\,{#2}\right)}
\newcommand{\dK}[2]{\operatorname{d}_{\operatorname{K}}\left({#1}\,;\,{#2}\right)}

\usepackage[colorlinks,linkcolor=black,citecolor=black,urlcolor=black]{hyperref}

\title[On the size of chaos via Glauber calculus]{On the size of chaos via Glauber calculus\\in the classical mean-field dynamics}

\author[M. Duerinckx]{Mitia Duerinckx}
\address[Mitia Duerinckx]{Université Paris-Saclay, CNRS, Laboratoire de Mathématiques d'Orsay, 91400~Orsay, France \& Universit\'e Libre de Bruxelles, Département de Mathématique, 1050~Brussels, Belgium}
\email{mduerinc@ulb.ac.be}

\begin{document}
\selectlanguage{english}

\maketitle

\begin{abstract}
We consider a system of classical particles, interacting via a smooth, long-range potential, in the mean-field regime, and we optimally analyze the propagation of chaos in form of sharp estimates on many-particle correlation functions. While approaches based on the BBGKY hierarchy are doomed by uncontrolled losses of derivatives, we propose a novel non-hierarchical approach that focusses on the empirical measure of the system and exploits discrete stochastic calculus with respect to initial data in form of higher-order Poincaré inequalities for cumulants. This main result allows to rigorously truncate the BBGKY hierarchy to an arbitrary precision on the mean-field timescale, thus justifying the Bogolyubov corrections to mean field. As corollaries, we also deduce a quantitative central limit theorem for fluctuations of the empirical measure, and we discuss the Lenard--Balescu limit for a spatially homogeneous system away from thermal equilibrium.
\end{abstract}

\setcounter{tocdepth}{1}
\tableofcontents

\section{Introduction}
\subsection{General overview}
We consider the dynamics of an interacting system of $N$ classical particles in the ambient space $\T^d$, as given by the following Newton's equations of motion, for $1\le j\le N$,
\begin{gather}\label{eq:part-dyn}
\frac{d}{dt}x_{j,N}=v_{j,N},\qquad\frac{d}{dt}v_{j,N}=-\frac1N\sum_{1\le l\le N\atop l\ne j}\nabla V(x_{j,N}-x_{l,N}),\\
x_{j,N}|_{t=0}=x_j^\circ,\qquad v_{j,N}|_{t=0}=v_j^\circ,\nonumber
\end{gather}
where $\{(x_{j,N},v_{j,N})\}_{j=1}^N$ denotes the set of positions and velocities of the particles in the phase space $\D:=\T^d\times\R^d$, where $V:\T^d\to\R$ is a long-range interaction potential, and where the mean-field scaling is considered.
In the regime of a large number $N\gg1$ of particles, we naturally focus on a statistical description of the system and consider the evolution of a random ensemble of particles.
In terms of a probability density $F_N$ on the $N$-particle phase space~$\D^N:=(\T^d\times\R^d)^N$,
Newton's equations~\eqref{eq:part-dyn} are equivalent to the following Liouville equation,
\begin{equation}\label{eq:Liouville}
\partial_t F_N+\sum_{j=1}^Nv_j\cdot\nabla_{x_j}F_N=\frac1N\sum_{1\le j\ne l\le N}\nabla V(x_j-x_l)\cdot\nabla_{v_j}F_N.
\end{equation}
Exchangeability of the particles translates into the symmetry of $F_N$ in its $N$ variables $z_j:=(x_j,v_j)\in\D$, $1\le j\le N$.
For simplicity, particles are assumed to be initially chaotic, that is, initial data $\{z_{j}^\circ:=(x_j^\circ,v_j^\circ)\}_{j=1}^N$ are independent and identically distributed (iid) with some common phase-space density $F^\circ:\D\to\R^+$. This means that $F_N$ is initially tensorized,
\begin{equation}\label{eq:init-tens}
F_N|_{t=0}=(F^\circ)^{\otimes N}.
\end{equation}
In the large-$N$ limit,
one looks for an averaged description of the system, e.g.\@ focussing on the evolution of one ``typical'' particle, as described by the first marginal of~$F_N$,
\[F_N^{1}(z):=\int_{\D^{N-1}}F_N(z,z_2,\ldots,z_N)\,dz_2\ldots dz_N.\]
Neglecting the correlations between particles (in view of the so-called Boltzmann's chaos assumption) formally leads to the following mean-field approximation: $F_N^1$ is expected to remain close to the solution~$F$ of the Vlasov equation,
\begin{gather}\label{eq:Vlasov}
\partial_tF+v\cdot\nabla_xF=(\nabla V\ast F)\cdot\nabla_vF,\\
(\nabla V\ast F)(x):=\int_{\D}\nabla V(x-y)\,F(y,v)\,dy\,dv,\nonumber
\end{gather}
with initial data $F|_{t=0}=F^\circ$.
We refer e.g.\@ to~\cite{Golse-rev2} for a review of rigorous results on this well-travelled topic.

\medskip
Corrections to this mean-field approximation stem from correlations and are easily unravelled by means of the BBGKY approach, as we briefly recall.
For $1\le m\le N$, we define the $m$-particle density $F_N^m$ as the $m$-th marginal of $F_N$,
\begin{equation}\label{eq:def-mpartdens}
F_N^{m}(z_1,\ldots,z_m):=\int_{\D^{N-m}}F_N(z_1,\ldots,z_m,z_{m+1},\ldots,z_N)\,dz_{m+1}\ldots dz_N.
\end{equation}
The Liouville equation~\eqref{eq:Liouville} is then equivalent to the following so-called BBGKY hierarchy of equations for marginals,
\begin{multline}\label{eq:BBGKY}
\partial_t F_N^{m}+\sum_{j=1}^mv_j\cdot\nabla_{x_j}F_N^{m}=\frac1N\sum_{1\le j\ne l\le m}\nabla V(x_j-x_l)\cdot\nabla_{v_j}F_N^{m}\\
+\frac{N-m}N\sum_{j=1}^m\int_\D\nabla V(x_j-x_*)\cdot\nabla_{v_j}F_N^{m+1}(z_1,\ldots,z_m,z_*)\,dz_*,
\end{multline}
with the convention $F_N^{N+1}\equiv0$. The first right-hand side term in this formulation is precisely the one that breaks the tensorized structure~\eqref{eq:init-tens}: it creates correlations between initially independent particles and deviates from the mean-field theory.
As this term is of order~$O(\frac{m^2}N)$, the correction to the chaotic mean-field approximation \mbox{$F_N^m\to F^{\otimes m}$} is expected of the same order.
While neglecting the $2$-particle correlation function $G_N^2:=F_N^2-(F_N^1)^{\otimes2}$ turns the equation for $F_N^1$ into the Vlasov equation~\eqref{eq:Vlasov}, the first-order correction amounts to keeping into account the contribution of $G_N^2$ (of expected order~$O(\frac1N)$), and only neglecting the $3$-particle correlation function $G_N^3$, which is expected of smaller order. Similarly performing this truncation to higher orders would yield a description with finer accuracy, as first predicted by Bogolyubov~\cite{Bogolyubov-46}.
The justification requires fine a priori estimates on many-particle correlations: as predicted by physicists~\cite{Bogolyubov-46}, the $(m+1)$-particle correlation function $G_N^{m+1}$ is expected of order
\begin{equation}\label{eq:expected-order}
G_N^{m+1}=O(\tfrac1{N^{m}}),
\end{equation}
which is indeed consistent with the BBGKY equations.
In contrast with the usual notion of propagation of chaos~\cite{Kac-56}, which boils down to the convergence of marginals \mbox{$F_N^m\to F^{\otimes m}$}, such estimates~\eqref{eq:expected-order} provide a much finer description of the decorrelation between particles, and their proof was still remaining as an open problem.
The main difficulty is as follows: tracking fine estimates by iteratively solving the BBGKY hierarchy~\eqref{eq:BBGKY} leads to an uncontrolled loss of derivatives in velocity variables. This is in sharp contrast with the quantum mean-field setting, as well as with the Kac model and the ``soft spheres'' model, where no derivative is lost, so that a straightforward analysis of the corresponding BBGKY hierarchy is known to lead to optimal correlation estimates~\cite{PPS-19}.
The main goal of the present work is to provide a novel non-hierarchical approach that covers the classical mean-field setting.

\subsection{Main result}
We start with the proper definition of many-particle correlation functions, which are suitable polynomial combinations of marginals of~$F_N$.
For $1\le m\le N$,
the $m$-particle correlation function is defined by
\begin{equation}\label{eq:def-corr}
G_{N}^m(z_1,\ldots,z_m):=\sum_{\pi\vdash[m]}(|\pi|-1)!\,(-1)^{|\pi|-1}\prod_{B\in\pi}F_{N}^{|B|}(z_B),
\end{equation}
where $\pi$ runs through the list of all partitions of the index set $[m]:=\{1,\ldots, m\}$, where $B$ runs through the list of blocks of the partition $\pi$, where $|\pi|$ is the number of blocks in the partition, where $|B|$ is the cardinality of $B$, and where for $B=\{i_1,\ldots,i_l\}\subset[m]$ we write $z_B:=(z_{i_1},\ldots,z_{i_l})$.
In particular,
\begin{eqnarray*}
G_N^2&:=&F_N^2-(F_N^1)^{\otimes2},\\
G_N^3&:=&\Sym\big(F_N^3-3 F_N^2\otimes F_N^1+2\,(F_N^1)^{\otimes3}\big),\\
G_N^4&:=&\Sym\big(F_N^4-4\,F_N^3\otimes F_N^1-3\,F_N^2\otimes F_N^2+12\,F_N^2\otimes(F_N^1)^{\otimes2}-6\,(F_N^1)^{\otimes4}\big),
\end{eqnarray*}
and so on, where $\Sym$ stands for the symmetrization of coordinates.\footnote{More precisely, for $H:\D^m\to\R$, we write
$\Sym(H)(z_1,\ldots,z_m)=\frac1{m!}\sum_{\sigma\in\mathcal S_m}H(z_{\sigma(1)},\ldots,z_{\sigma(m)})$,
where $\mathcal S_m$ denotes the set of all permutations of the set $[m]$.}
The full distribution~$F_N$ is then recovered from correlation functions in form of a cluster expansion,
\begin{equation}\label{eq:cluster}
F_N(z_1,\ldots,z_N)=\sum_{\pi\vdash[N]}\prod_{B\in\pi}G_N^{|B|}(z_B).
\end{equation}
Together with the property that $\int_\D G_N^m(z_1,\ldots,z_m)\,dz_l=0$ for all $1\le l\le m$, this cluster expansion actually defines the correlation functions~\eqref{eq:def-corr} uniquely.

\medskip
Our main result gives a priori estimates with the optimal expected order~\eqref{eq:expected-order} in suitable negative Sobolev norms.
The strategy is as follows: since initial data are chaotic, cf.~\eqref{eq:init-tens}, they satisfy strong concentration properties, which can only be mildly deformed under Newton's flow as the mean-field scaling entails weak interactions. This is controlled by a deterministic Grönwall argument for particle trajectories, while concentration properties are exploited in form of functional inequalities in terms of so-called Glauber calculus with respect to initial data, cf.~Section~\ref{sec:iid}.
More precisely, we rely on new higher-order Poincaré inequalities for cumulants in the spirit of~\cite{Nourdin-Peccati-10}.
We believe that similar ideas could be useful for other types of systems.

\begin{theor1}[Optimal a priori estimates on correlations]\label{th:cumulant}
Assume that the interaction kernel $V$ is smooth and even, let $F^\circ\in \Pc\cap C^\infty_c(\D)$, let $F_N$ denote the solution of the Liouville equation~\eqref{eq:Liouville} with chaotic data~\eqref{eq:init-tens}, and let $\{G_N^{m}\}_{m=1}^N$ denote the corresponding correlation functions.
Then, for $0\le m\le N-1$, the $(m+1)$-particle correlation function $G_N^{m+1}$ is of order $O(N^{-m})$ in the following sense, for all $t\ge0$,
\begin{equation}\label{eq:bnd-cumulant}
\|G_N^{m+1;\,t}\|_{W^{-2m,1}(\D^{m+1})}\,\le\,\frac1{N^{m}}C_me^{C_mt},
\end{equation}
where the constant $C_m$ only depends on $d$, $m$, $\|\nabla V\|_{W^{m,\infty}(\T^d)}$, $\int_\D|v|^{2m}dF^\circ(z)$.
\end{theor1}

\begin{rems}$ $\label{rem:cum}
\begin{enumerate}[(i)]
\item {\it Exponential time growth:}\\
While optimal in terms of $N$-scaling, the above estimates suffer from an exponential time growth that originates in the Grönwall argument to control the correlation of particles along Newton's flow. For this reason, our conclusions are limited to the timescale $t\ll \log N$.
Extending such estimates to much longer times should require drastically different tools: rather than propagation of initial chaos, this would involve some relaxation mechanism and constitutes a major open problem in the field.
In the simpler setting of fluctuations around thermal equilibrium, an orthogonality argument allows to deduce time-uniform estimates on linearized correlation functions (although with slightly suboptimal $N$-scaling); this was first exploited by Bodineau, Gallagher, and Saint-Raymond in~\cite[Proposition~4.2]{BGSR-16}, see also~\cite[Lemma~2.2]{DSR-1}.
\item {\it Negative Sobolev norms:}\\
Negative Sobolev norms appear naturally in the proof since higher-order correlation functions are viewed as suitable higher-order finite differences.
However, note that for all $k\ge1$ the initial $W^{k,1}$ regularity of marginals is propagated by the Liouville equation uniformly in $N$,
\[\qquad\|F_N^{m;t}\|_{W^{k,1}(\D^m)}\,\lesssim\,e^{C_mt\|\nabla V\|_{W^{k,\infty}(\T^d)}}\|(F^{\circ})^{\otimes m}\|_{W^{k,1}(\D^m)},\]
and thus, by interpolation, we can deduce a corresponding bound~\eqref{eq:bnd-cumulant} on correlations in any smooth norm for smooth enough initial data $F^\circ$, at the expense of loosing a tiny power of the rate. The same comment applies to all subsequent corollaries.
\qedhere
\end{enumerate}
\end{rems}

\subsection{Applications}
We now turn to various applications of the above main result. After recovering the standard mean-field result, we justify Bogolyubov corrections, we establish a quantitative central limit theorem (CLT) for fluctuations of the empirical measure, and we discuss the so-called Lenard--Balescu limit.

\subsubsection{Mean field}
An optimal error estimate is recovered for the mean-field approximation~\eqref{eq:Vlasov}: we start from the BBGKY hierarchy~\eqref{eq:BBGKY} and we neglect $2$-particle correlations by means of the a priori bound~\eqref{eq:bnd-cumulant} on $G_N^2$.
Recall however that a simpler proof of this standard result follows from the Klimontovich approach~\cite{Klimontovich-67,Dobrushin-79}: the empirical measure associated with the particle dynamics is an exact (distributional) solution of the Vlasov equation and the mean-field approximation reduces to a stability question (see e.g.~\cite{Golse-rev2}).

\begin{cor1}[Mean field]\label{cor:MFL}
Let the same assumptions hold as in Theorem~\ref{th:cumulant}.
Then, the $1$-particle density~$F_N^1$ is close to the solution of the Vlasov equation~\eqref{eq:Vlasov} in the following sense, for all $t\ge0$ and $\delta>0$,
\[\|F_N^{1;t}-F^t\|_{W^{-2-\delta,1}(\D)}\,\le\,\frac1N\,C_\delta e^{C_\delta t^{1+\delta}},\]
where the constant $C_\delta$ only depends on $d,\delta$, $\|\nabla V\|_{W^{2+\delta,\infty}(\T^d)}$, $\|F^\circ\|_{\Ld^{1+\delta}(\D)}$, $\supp F^\circ$.
\end{cor1}

\begin{rem}
In this result, we assume for simplicity that the initial density $F^\circ$ is compactly supported. Up to an approximation argument, this can however be relaxed into e.g.\@ an exponential decay assumption, at the expense of loosing a tiny power of the rate~$O(\frac1N)$. The same comment applies to all subsequent corollaries.
\end{rem}

\subsubsection{Bogolyubov corrections to mean field}
While the above mean-field result is classical, our new a priori bounds on correlations allow to truncate the BBGKY hierarchy~\eqref{eq:BBGKY} to any accuracy.
As first predicted in~\cite{Bogolyubov-46}, the next-order correction to mean field is governed by the $2$-particle correlation function and takes form of the following closed system for~$F_N^1$ and $NG_N^2$, which is known as the Bogolyubov equations,
\begin{eqnarray}
\partial_tF_N^1+v\cdot\nabla_xF_N^1&=&\tfrac{N-1}N(\nabla V\ast F_N^1)\cdot\nabla_vF_N^1\label{eq:Bogo0}\\
&&+\,\tfrac1N\int_\D\nabla V(x-x_*)\cdot\nabla_v(NG_N^2)(z,z_*)\,dz_*+O(\tfrac1{N^2}),\nonumber\\
\partial_t(NG_N^2)+iL_{F_N^1}(NG_N^2)&=&\nabla V(x_1-x_2)\cdot(\nabla_{v_1}-\nabla_{v_2})(F_N^1\otimes F_N^1)+O(\tfrac1{N})\nonumber\\
&&-\,\big(\nabla V\ast F_N^1(x_1)\cdot\nabla_{v_1}+\nabla V\ast F_N^1(x_2)\cdot\nabla_{v_2}\big)(F_N^1\otimes F_N^1),\nonumber
\end{eqnarray}
with initial data $F_N^1|_{t=0}=F^\circ$ and $(NG_N^2)|_{t=0}=0$, where $iL_{F}$ stands for the $2$-particle linearized Vlasov operator at $F$,
\begin{multline}\label{eq:LFH}
iL_{F}H\,:=\,\big(v_1\cdot\nabla_{x_1}+v_2\cdot\nabla_{x_2}\big)H
-\big(\nabla V\ast F(x_1)\cdot\nabla_{v_1}+\nabla V\ast F(x_2)\cdot\nabla_{v_2}\big)H\\
-\nabla_{v_1}F (z_1)\cdot\int_\D\nabla V(x_1-x_*)H(z_2,z_*)\,dz_*\\
-\nabla_{v_2}F(z_2)\cdot\int_\D\nabla V(x_2-x_*)H(z_1,z_*)\,dz_*.
\end{multline}
A rigorous justification of this Bogolyubov correction is obtained as an application of Theorem~\ref{th:cumulant}.
Since the correction is small on
the short timescale $t\ll\log N$, to which we are anyway restricted due to the time growth in our bounds, we note that $F_N^1$ can be replaced by its Vlasov approximation $F$ in the equation for $NG_N^2$, cf.~\eqref{eq:H-defin} below.
In the simplified setting of fluctuations close to thermal equilibrium, a similar result is contained in~\cite[Section~4]{DSR-1}.
The extension to higher order is straightforward and omitted.

\begin{cor1}[Bogolyubov corrections]\label{cor:Bogo}
Let the same assumptions hold as in Theorem~\ref{th:cumulant}.
Denote by $F$ the solution of the Vlasov equation~\eqref{eq:Vlasov},
and let $H_N^1$ satisfy the following corrected Vlasov equation,
\begin{multline}\label{eq:HN1-defin}
\partial_tH_N^1+v\cdot\nabla_xH_N^1\,=\,\frac{N-1}N(\nabla V\ast H_N^1)\cdot\nabla_vH_N^1\\
+\frac1N\int_\D\nabla V(x-x_*)\cdot\nabla_vH^2(z,z_*)\,dz_*,
\end{multline}
where the correction $H^2$ is the solution of
\begin{multline}\label{eq:H-defin}
\partial_tH^2+iL_{F}H^2\,=\,\nabla V(x_1-x_2)\cdot(\nabla_{v_1}-\nabla_{v_2})(F\otimes F)\\
-\big(\nabla V\ast F(x_1)\cdot\nabla_{v_1}+\nabla V\ast F(x_2)\cdot\nabla_{v_2}\big)(F\otimes F),
\end{multline}
with initial data $H_N^1|_{t=0}=F|_{t=0}=F^\circ$ and $H^2|_{t=0}=0$,
where $iL_F$ is the $2$-particle linearized Vlasov operator at~$F$, cf.~\eqref{eq:LFH}.
Then, the $1$-particle density is close to $H_N^1$ to next order in the following sense, for all $t\ge0$ and $\delta>0$,
\[1\wedge\|F_N^{1;t}-H_N^{1;t}\|_{W^{-4-\delta,1}(\D)}\,\le\,\frac1{N^2}C_\delta e^{C_\delta t^{1+\delta}},\]
where the constant $C_\delta$ only depends on $d,\delta$, $\|\nabla V\|_{W^{8+\delta,\infty}(\T^d)}$, $\|F^\circ\|_{\Ld^{1+\delta}(\R^d)}$, $\supp F^\circ$.
\end{cor1}

While the equation~\eqref{eq:H-defin} for the Bogolyubov correction $H^2\sim NG_N^2$ looks rather cumbersome, a simpler reformulation is provided in~\eqref{eq:var-emp-lim} below in terms of fluctuations of the empirical measure associated with the particle dynamics~\eqref{eq:part-dyn},
\begin{equation}\label{eq:emp-meas}
\mu_N^t\,:=\,\frac1N\sum_{j=1}^N\delta_{(x_{j,N}^t,v_{j,N}^t)}.
\end{equation}
More precisely, in the spirit of the Klimontovich theory~\cite{Klimontovich-67}, $H^2$ is reformulated in terms of the linearized Vlasov operator applied to initial fluctuations of the empirical measure.

\subsubsection{CLT for the empirical measure}
As the empirical measure~\eqref{eq:emp-meas} is a (distributional) solution of the Vlasov equation, fluctuations are expected to satisfy the corresponding linearized equation.
While a qualitative CLT in this flavor was first established in the early work of Braun and Hepp~\cite{Braun-Hepp-77} (see also~\cite[Section~I.7.5]{Spohn-91} and~\cite{Lancellotti-09}), we improve it into an optimal quantitative statement. In fact, this result is essentially equivalent to the above Bogolyubov correction to mean field.

\begin{cor1}[CLT for empirical measure]\label{cor:CLT}
Let the same assumptions hold as in Theorem~\ref{th:cumulant}. Denote by $\G^\circ$ the Gaussian field describing the fluctuations of the initial empirical measure, in the sense that $\sqrt N\int_\D\phi\,d(\mu_N^\circ-F^\circ)$ converges in law to $\int_\D\phi\,\G^\circ$ for all $\phi\in C^\infty_c(\D)$.\footnote{More explicitly, $\G^\circ$ is the centered Gaussian field characterized by its variance structure $\varr{\int_\D\phi\, \G^\circ}=\int_\D\phi^2F^\circ-(\int_\D\phi F^\circ)^2$ for all $\phi\in C^\infty_c(\D)$.}
Denote by $F$ the solution of the Vlasov equation~\eqref{eq:Vlasov}, and for all $J^\circ\in C^\infty_c(\D)$ denote by $U[J^\circ]$ the solution of the linearized equation,
\begin{equation}\label{eq:lin-Vlas-fl}
\partial_tU[J^\circ]+v\cdot\nabla_xU[J^\circ]\,=\,(\nabla V\ast F)\cdot\nabla_vU[J^\circ]+(\nabla V\ast U[J^\circ])\cdot\nabla_vF,
\end{equation}
with initial data $U[J^\circ]|_{t=0}=J^\circ$.
Then, for all $\phi\in C^\infty_c(\D)$ and $t\ge0$, the random variable
$\sqrt N\int_\D\phi\,d(\mu_N^t-F^t)$ converges in law to the Gaussian random variable $\int_\D\phi\,U^t[\G^\circ]$.
The limiting variance is alternatively reformulated as
\begin{equation}\label{eq:var-emp-lim}
(\sigma_\phi^t)^2:=\varr{\int_\D\phi\,U^t[\G^\circ]}\,=\,\int_{\D^2}(\phi\otimes\phi)\,H^{2;t}+\bigg(\int_\D\phi^2F^t-\Big(\int_\D\phi F^t\Big)^2\bigg),
\end{equation}
where $H^2$ denotes the Bogolyubov correction defined in Corollary~\ref{cor:Bogo}. In addition, provided that $\sigma_\phi^t\ne0$, the following optimal quantitative estimate holds
for all~$t\ge0$ and $\delta>0$,
\begin{multline}\label{eq:CLT-fin}
\dW{\sqrt N\int_\D\phi\,d(\mu_N^t-F^t)}{\sigma_\phi^t\Nc}+\dK{\sqrt N\int_\D\phi\,d(\mu_N^t-F^t)}{\sigma_\phi^t\Nc}\\
\,\le\,\frac1{\sqrt N}C_{\delta,\phi} e^{C_{\delta,\phi} t^{1+\delta}},
\end{multline}
where $\dW\cdot\cdot$ and $\dK\cdot\cdot$ denote the $1$-Wasserstein and Kolmogorov distances, where $\Nc$ denotes a standard Gaussian random variable, and where the constant $C_{\delta,\phi}$ only depends on $d,\delta$, $\|\phi\|_{W^{3+\delta,\infty}(\D)}$, $\|\nabla V\|_{W^{3+\delta,\infty}(\T^d)}$, $\|F^\circ\|_{\Ld^{1+\delta}(\D)}$, $\supp F^\circ$.
\end{cor1}

\subsubsection{Lenard--Balescu limit}
We consider the important particular case of a spatially homogeneous system, that is, $F^\circ(x,v)\equiv f^\circ(v)$. The mean-field force then obviously vanishes by symmetry and the Boglyubov correction becomes the relevant leading order.
As this correction will play a role on long timescales only, we naturally filter out oscillations created by spatial transport on shorter times, hence we focus on the projection on the kernel of the transport, that is, the velocity distribution
\begin{equation}\label{eq:vel-distr}
f_N^1(v):=\int_{\T^d}F_N^1(x,v)\,dx,
\end{equation}
which satisfies the following simplified version of the Bogolyubov equations~\eqref{eq:Bogo0},
\begin{gather}
\partial_tf_N^1\,=\,\frac1N\int_{\T^d}\int_\D\nabla V(x-x_*)\cdot\nabla_{v}(NG_N^2)(z,z_*)\,dz_*\,dx\,+O(\tfrac1{N^2}),\label{eq:Bogo}\\
\partial_t(NG_N^2)+i L_{f_N^1}(NG_N^2)\,=\,\nabla V(x_1-x_2)\cdot(\nabla_{v_1}-\nabla_{v_2})(f_N^1\otimes f_N^1)\,+O(\tfrac1{N}),\nonumber
\end{gather}
with initial data $f_N^1|_{t=0}=f^\circ$ and $(NG_N^2)|_{t=0}=0$, where the linearized Vlasov operator now takes the reduced form
\begin{multline*}
i L_{f}H=\big(v_1\cdot\nabla_{x_1}+v_2\cdot\nabla_{x_2}\big)H
-\nabla f(v_1)\cdot\int_\D\nabla V(x_1-x_*)H(z_2,z_*)\,dz_*\\
-\nabla f(v_2)\cdot\int_\D\nabla V(x_2-x_*)H(z_1,z_*)\,dz_*.
\end{multline*}
As the Bogolyubov correction is given by $2$-particle correlations, it describes collisions and is expected to lead to irreversible effects.
This is however difficult to grasp from~\eqref{eq:Bogo} since in particular the Bogolyubov correction is not Markovian: solving the equation for $NG_N^2$ requires to know the whole history of $f_N^1$.
While the $O(\frac1N)$ Bogolyubov correction in~\eqref{eq:Bogo} is expected to have a $O(1)$ contribution only on the relevant long timescale~$t\sim N$, the $2$-particle correlation function $G_N^2$ evolves on the short timescale~$t\sim1$ and is thus expected to relax. This relaxation is a consequence of linear Landau damping for two typical particles; it amounts to approximating collisions as instantaneous events, thereby neglecting memory effects. More precisely, the time-rescaled $1$-particle velocity density $f_N^{1;Nt}$ is predicted to remain close to the solution~$f$ of the following so-called Lenard--Balescu kinetic equation,
\begin{align}\label{eq:LB}
\partial_tf=\LB(f):=\nabla\cdot\int_{\R^d}B(v,v-v_*;\nabla f)\,\big(f_*\nabla f-f\nabla_*f_*\big)\,dv_*,
\end{align}
with the notation $f=f(v)$, $f_*=f(v_*)$, $\nabla=\nabla_v$, and $\nabla_*=\nabla_{v_*}$,
in terms of the collision kernel
\begin{align}\label{eq:LB-kernel}
B(v,v-v_*;\nabla f):=\sum_{k\in2\pi\Z^d}(k\otimes k)\,\pi\widehat V(k)^2\tfrac{\delta(k\cdot(v-v_*))}{|\e(k,k\cdot v;\nabla f)|^2}\,dk,
\end{align}
and of the dispersion function
\begin{align}\label{eq:LB-disp}
\e(k,k\cdot v;\nabla f):=1+\widehat V(k)\int_{\R^d}\tfrac{k\cdot\nabla f(v_*)}{k\cdot(v-v_*)-i0}\,dv_*.
\end{align}

This Lenard--Balescu equation was formally derived in the early 60s independently by Guernsey~\cite{Guernsey-60,Guernsey-62}, Lenard~\cite{Lenard-60}, and Balescu~\cite{Balescu-60,BT-61} in the context of plasma physics.
At a formal level, it preserves mass, momentum, and kinetic energy, it admits Maxwellian distributions as stationary solutions, and it satisfies an $H$-theorem,
\[\partial_t\int_{\R^d} f\log f=-\iint_{\R^d\times\R^d}\big((\nabla-\nabla_*)\sqrt{ff_*}\big)\cdot B(v,v-v_*;\nabla f)\big((\nabla-\nabla_*)\sqrt{ff_*}\big)~\le~0,\]
hence it describes the relaxation of the velocity density towards Maxwellian equilibrium on the relevant timescale;
we refer to~\cite[Chapter~5]{Nicholson} for a thorough physics discussion.
A key feature is the nonlocal nonlinearity of the kernel~\eqref{eq:LB-kernel}, taking into account collective effects in form of nonlocal dynamical screening.
Due to this full nonlinearity, the mathematical study of the equation is reputedly difficult.

\medskip
Apart from some partial attempts in~\cite{Lancellotti-10,VW-18} (see also~\cite{BPS-13,Winter-19}), any rigorous derivation from particle dynamics has remained elusive.
More recently, in the simplified setting of fluctuations around thermal equilibrium, we obtained in~\cite{DSR-1} with Laure Saint-Raymond a rigorous justification of the linearized Lenard--Balescu equation, although restricted to an intermediate timescale $t\sim N^r$ with $r<1$. The analysis pointed out three key difficulties:
\begin{enumerate}[(a)]
\item the validity of sharp bounds on many-particle correlation functions up to the relevant timescale $t\sim N$;
\item the long-time control of some resonances related to plasma echoes;
\item the well-posedness of the Lenard--Balescu equation, which requires a dynamic control of the dispersion function.
\end{enumerate}
While the present work provides sharp correlation estimates away from equilibrium, which were a missing ingredient in~\cite{DSR-1}, these only hold on an even shorter timescale $t\ll \log N$, cf.~Remark~\ref{rem:cum}(i), and the required extension~(a) is left as a major open problem. Next, difficulty~(b) is easily shown to vanish on such a logarithmic timescale. Finally, as no evolution occurs for $t\ll N$, we are simply led to the Lenard--Balescu operator applied to the initial data, instead of a genuine evolution equation, so that difficulty~(c) also disappears.
In this setting, repeating a similar analysis as in~\cite{DSR-1}, now starting from Corollary~\ref{cor:Bogo}, we obtain the following nonlinear extension of~\cite{DSR-1} away from equilibrium.

\begin{samepage}
\begin{cor1}[Lenard--Balescu limit]\label{cor:LB}
Let the same assumptions hold as in Theorem~\ref{th:cumulant}. Further assume that
\begin{enumerate}[$\bullet$]
\item the initial density $F^\circ$ is
\begin{itemize}
\item spatially homogeneous (that is, $F^\circ(x,v)\equiv f^\circ(v)$);
\item linearly Vlasov-stable (that is, for any direction $k$, the projected initial density $\pi_k^\circ(y):=\int_{\R^d}\delta(y-\frac{k\cdot v}{|k|})\,f^\circ(v)\,dv$ satisfies $y(\pi_k^\circ)'(y)\le0$ for all $y$);
\end{itemize}
\item the interaction potential $V:\T^d\to\R$ is
\begin{itemize}
\item positive definite (that is, $\widehat V\ge0$);
\item small enough (that is, $\|V\|_{\Ld^\infty(\T^d)}\le\frac1{C_0}$ for some large enough constant $C_0$ only depending on the initial density $F^\circ$ via $\|f^\circ\|_{W^{2+\delta,1}(\R^d)}$ for any $\delta>0$).
\end{itemize}
\end{enumerate}
Then, given $\delta>0$, for any sequence $(t_N)_N$ with $1\ll t_N\ll (\log N)^{1-\delta}$ (that is, $t_N\to\infty$ and $\frac{t_N}{(\log N)^{1-\delta}}\to0$),
the $1$-particle velocity density $f_N^1$ (cf.~\eqref{eq:vel-distr}) satisfies
\[\lim_{N\uparrow\infty}N(\partial_tf_N^{1,t})_{t=t_N\tau}=\LB(f^\circ),\]
as a function of $(\tau,v)$ in the weak sense of $\Dc'(\R^+\times\R^d)$,
where we recall that the Lenard--Balescu operator $\LB$ is defined in~\eqref{eq:LB}.
\end{cor1}
\end{samepage}

\subsection*{Plan of the article}
The article is organized as follows. The proof of Theorem~\ref{th:cumulant} is split into Sections~\ref{sec:iid}, \ref{sec:MF}, and~\ref{sec:th1}. In order to avoid hierarchical arguments, in the spirit of the Klimontovich approach~\cite{Klimontovich-67}, we note that correlation functions are equivalent to cumulants of the empirical measure. In Section~\ref{sec:iid}, we introduce so-called Glauber calculus with respect to iid random initial data and we establish new higher-order Poincaré inequalities for cumulants.
In Section~\ref{sec:MF}, we show how such inequalities are deformed under Newton's flow: by means of a Grönwall argument, we estimate how the trajectory of a given particle is sensitive to modifications of initial data of other particles. Combining these results, we conclude with the proof of Theorem~\ref{th:cumulant} in Section~\ref{sec:th1}.
Next, we turn to the applications, and Corollaries~\ref{cor:MFL}, \ref{cor:Bogo}, \ref{cor:CLT}, \ref{cor:LB} are established in Sections~\ref{sec:MFL}, \ref{sec:Bogo}, \ref{sec:CLT}, \ref{sec:LB}, respectively.

\subsection*{Notation}
\begin{enumerate}[$\bullet$]
\item We denote by $C\ge1$ any constant that only depends on the space dimension $d$.
We use the notation $\lesssim$ (resp.~$\gtrsim$) for $\le C\times$ (resp.~$\ge\frac1C\times$) up to such a multiplicative constant~$C$. We add subscripts to $C,\lesssim,\gtrsim$ to indicate dependence on other parameters.
\item Initial data $(z_j^\circ=(x_j^\circ,v_j^\circ))_j$ are iid random variables with law $F^\circ$ on the phase space \mbox{$\D=\T^d\times\R^d$}, constructed on a probability space~$(\Omega^\circ,\Pm)$. We denote by $\E[\cdot]$ the expectation with respect to this probability ensemble, by $\Var[\cdot]$ the variance, and by~$\kappa_m^\circ[\cdot]$ the $m$-th cumulant, cf.~\eqref{eq:def-kappa-m} below.
We denote by $D^\circ$ the Glauber gradient and by $\Lc^\circ$ the Glauber Laplacian on~$\Ld^2(\Omega^\circ)$ as defined in Section~\ref{sec:iid} below.
\item We denote by $F_N$ the probability density on the $N$-particle phase space $\D^N=(\T^d\times\R^d)^N$, we write~$F_N^m$ for the $m$-th marginal or $m$-particle density, cf.~\eqref{eq:def-mpartdens}, and we denote by $G_N^m$ the $m$-particle correlation function, cf.~\eqref{eq:def-corr}.
\item For $m\ge0$ we set $[m]=\{1,\ldots,m\}$, and for an index set $E=\{i_1,\ldots,i_l\}$ we write $z_E=(z_{i_1},\ldots,z_{i_l})$. Given an index set $E$, the notation $\pi\vdash E$ indicates that $\pi$ is a partition of $E$. When writing
\[\sum_{\pi\vdash E}\prod_{B\in\pi}f(|\pi|,B),\]
the sum thus runs over all partitions $\pi$ of the index set $E$ and the product runs over all blocks $B$ of the partition $\pi$, while $|\pi|$ denotes the number of blocks in the partition $\pi$ and $|B|$ denotes the cardinality of $B$.
\item For a measurable function $h\in\Ld^1(\T^d)$, we denote by $\widehat h\in\ell^\infty(2\pi\Z^d)$ its Fourier coefficients, $\widehat h(k)=\int_{\T^d}e^{-ik\cdot x}h(x)\,dx$.
We denote by $k_j\in2\pi\Z^d$ the Fourier conjugate variable associated with $x_j\in\T^d$.
Given a linear operator $A$ on $\Ld^1(\D)$, we denote by $\widehat A$ the corresponding operator acting in Fourier space, that is, $\widehat A\,\widehat h=\widehat{Ah}$.
\item For $a,b\in\R$ we write $a\wedge b:=\min\{a,b\}$, $a\vee b:=\max\{a,b\}$, and $\langle a\rangle:=(1+a^2)^{1/2}$.
\end{enumerate}

\smallskip
\section{Glauber calculus for iid random initial data}\label{sec:iid}
Let the iid data $(z_{j}^\circ=(x_{j}^\circ,v_{j}^\circ))_j$ be constructed on a given probability space $(\Omega^\circ,\Pm)$, and let the latter be endowed with the minimal $\sigma$-algebra generated by $(z_{j}^\circ)_{j}$.
For a random variable $Y=Y((z_{j}^\circ)_{j})\in\Ld^2(\Omega^\circ)$, we then define its so-called {\it Glauber derivative} with respect to the initial data $z_k^\circ$,
\begin{equation}\label{eq:def-D0}
D_k^\circ Y:=Y((z_{j}^\circ)_{j})-\E_k[Y((z_{j}^\circ)_{j})],
\end{equation}
where $\E_k$ denotes the integration with respect to the variable $z_k^\circ$ only. The Glauber gradient $D^\circ Y=(D_j^\circ Y)_j$ is an element of $\ell^\infty(\N;\Ld^2(\Omega^\circ))$ and measures the sensitivity of $Y$ with respect to the underlying data $(z_j^\circ)_j$.
In those terms, the celebrated Efron--Stein inequality~\cite{Efron-Stein-81} takes form of the following Poincaré inequality in the probability space.

\begin{lem}[Efron--Stein's inequality~\cite{Efron-Stein-81}]\label{prop:ES}
For all random variables $Y\in\Ld^2(\Omega^\circ)$, there holds
\[\var{Y}\,\le\,\expecM{\sum_j|D_j^\circ Y|^2}\,=\,\|D^\circ Y\|_{\ell^2(\N;\Ld^2(\Omega^\circ))}^2.\qedhere\]
\end{lem}

While this provides a useful control of the variance of functions of the random data, we show below that a similar control can be extended to higher-order cumulants in form of higher-order Poincaré inequalities. This extends to the iid setting a result previously established by Nourdin and Peccati~\cite{Nourdin-Peccati-10} in the Gaussian case by means of Malliavin calculus.
First recall that the $m$-th cumulant of a bounded random variable~$Y$ is defined by
\[\kappa^\circ_m[Y]:=\big((\tfrac{d}{dt})^m\log\expec{e^{tY}}\big)\big|_{t=0},\]
that is,
\begin{eqnarray*}
\kappa^\circ_1[Y]&=&\expecm{Y},\\
\kappa^\circ_2[Y]&=&\expecm{Y^2}-\expecm{Y}^2~=~\varm{Y},\\
\kappa^\circ_3[Y]&=&\expecm{Y^3}-3\,\expecm{Y^2}\expecm{Y}+2\,\expecm{Y}^3,\\
\kappa^\circ_4[Y]&=&\expecm{Y^4}-4\,\expecm{Y^3}\expecm{Y}-3\,\expecm{Y^2}^2+12\,\expecm{Y^2}\expecm{Y}^2-6\expec{Y}^4,
\end{eqnarray*}
and so on. The following general formula holds for all $m\ge1$,
\begin{equation}\label{eq:def-kappa-m}
\kappa^\circ_m[Y]\,=\,\sum_{\pi\vdash[m]}(-1)^{|\pi|-1}(|\pi|-1)!\prod_{B\in\pi}\expecm{Y^{|B|}},
\end{equation}
which can alternatively be formulated in terms of incomplete Bell polynomials.
Conversely, moments can be recovered from cumulants in form of a cluster expansion,
\[\expec{Y^m}=\sum_{\pi\vdash[m]}\prod_{B\in\pi}\kappa^\circ_{|B|}[Y].\]
Whereas the following simplified statement suffices for our purpose in this work, more precise estimates are obtained in the proof in form of exact representation formulas, cf.~Lemma~\ref{lem:rep-cum}; the proof is postponed to Section~\ref{sec:cumulants}.

\begin{theor}[Higher-order Poincaré inequalities for cumulants]\label{th:cum-gen}
For all bounded random variables $Y_N=Y_N((z_j^\circ)_{j=1}^N)$ depending only on $N$ initial data, there holds for all $m\ge1$,
\[\kappa^\circ_{m+1}(Y_N)\,\lesssim_{m}\,\sum_{k=0}^{m-1}N^{k+1}\sum_{a_1,\ldots,a_{m+1}\ge1\atop\sum_ja_j=m+k+1}\prod_{j=1}^{m+1}\|(D^\circ)^{a_j}Y_N\|_{\ell^\infty_{\ne}\big(\Ld^{\frac{1}{a_j}(m+k+1)}(\Omega^\circ)\big)},\]
where we use the following short-hand notation for norms of iterated Glauber derivatives,
\begin{equation}\label{eq:not-linft-Dp}
\|(D^\circ)^nY\|_{\ell^\infty_{\ne}(\Ld^p(\Omega^\circ))}\,:=\,\sup_{j_1,\ldots, j_n\atop\text{distinct}}\|D^\circ_{j_1}\ldots D^\circ_{j_n}Y\|_{\Ld^p(\Omega^\circ)}.\qedhere
\end{equation}
\end{theor}

The approximate normality of a random variable essentially follows from the convergence of the first two moments and from the smallness of higher cumulants. It is nicely quantified as follows, where the upper bound interestingly reduces to the above bound on the third cumulant only.
This result is known as a second-order Poincaré inequality for approximate normality; it was first established by Chatterjee~\cite[Theorem~2.2]{Chat08} based on Stein's method for the $1$-Wasserstein distance, while the corresponding bound on the Kolmogorov distance can be found in~\cite[Theorem~4.2]{LRP-15}.
A short argument for the $1$-Wasserstein distance is included in Section~\ref{sec:idea-Chat} below for completeness.

\begin{theor}[Second-order Poincaré inequality for approximate normality;~\cite{Chat08,LRP-15}]\label{th:2ndP}
For all bounded random variables $Y$, setting $\sigma_Y^2:=\var{Y}$, there holds
\begin{multline*}
\dW{\tfrac1{\sigma_Y}\big(Y-\expec{Y}\big)}{\Nc}+\dK{\tfrac1{\sigma_Y}\big(Y-\expec{Y}\big)}{\Nc}\\
\lesssim\frac1{\sigma_Y^{3}}\sum_{j}\expec{|D_j^\circ Y|^6}^\frac12+\frac1{\sigma_Y^2}\bigg(\sum_{j}\Big(\sum_{l}\expec{|D_l^\circ Y|^4}^\frac14\expec{|D_j^\circ D_l^\circ Y|^4}^\frac14\Big)^2\bigg)^\frac12,
\end{multline*}
where we recall that $\dW\cdot\Nc$ and $\dK\cdot\Nc$ stand for the $1$-Wasserstein and the Kolmogorov distances to a standard Gaussian random variable.
\end{theor}

\subsection{Glauber calculus}
For functions of independent random sequences, a discrete version of Malliavin calculus was first developed by Privault~\cite{Privault-97}; see also~\cite{Privault-Serafin-18} and references therein.
More recently, based on the above discrete gradient $D^\circ$, cf.~\eqref{eq:def-D0}, a related construction was performed by us in a joint work with Gloria and Otto~\cite[Lemmas~5.1]{DGO1}, and independently by Decreusefond and Halconruy in~\cite{Decreusefond-Halconruy-19}.
We refer to this discrete stochastic calculus as {\it Glauber calculus} and we briefly recall the theory below.

\medskip
A direct computation shows that $D_j^\circ$ is self-adjoint on~$\Ld^2(\Omega^\circ)$ and satisfies the following commutation relations, for $j\ne l$,
\begin{equation}\label{eq:commut-D0}
D^\circ_jD^\circ_j=D^\circ_j,\qquad D^\circ_jD^\circ_l=D^\circ_lD^\circ_j.
\end{equation}
We then construct the corresponding Glauber Laplacian (which plays the role of the Ornstein-Uhlenbeck operator in Malliavin calculus)
\[\Lc^\circ\,:=\,\sum_j(D^\circ_j)^*D^\circ_j\,=\,\sum_jD^\circ_j,\]
which is clearly densely defined on $\Ld^2(\Omega^\circ)$. Various properties of this fundamental operator are collected in the following.

\begin{lem}[Glauber Laplacian]\label{lem:OU}$ $
\begin{enumerate}[(i)]
\item $\Lc^\circ$ is essentially self-adjoint and nonnegative.
\item $\Lc^\circ$ has dense image in $\Ld^2(\Omega^\circ)/\R:=\{Y\in\Ld^2(\Omega^\circ):\expec{Y}=0\}$.
\item $\Lc^\circ$ has kernel reduced to constants, $\ker\Lc^\circ=\R$, and has a unit spectral gap. In addition, the spectrum of $\Lc^\circ$ is the set of natural numbers.
\item The restriction of $\Lc^\circ$ to $(\ker\Lc^\circ)^\bot=\Ld^2(\Omega^\circ)/\R$ admits a well-defined inverse $\Tc^\circ$, which is a self-adjoint nonnegative contraction on $\Ld^2(\Omega^\circ)/\R$.
\item The inverse operator $\Tc^\circ$ satisfies, for all $p\ge1$ and $Y\in\Ld^p(\Omega^\circ)$ with $\expec{Y}=0$,
\[\quad\|\Tc^\circ Y\|_{\Ld^p(\Omega^\circ)}\lesssim \tfrac{p^2}{p-1}\|Y\|_{\Ld^p(\Omega^\circ)}.\qedhere\]
\end{enumerate}
\end{lem}

\begin{proof}
Starting from the identity $\Id=\prod_j(D_j^\circ+\E_j)$ on $\Ld^2(\Omega^\circ)$, we are led to
\[\Id=\sum_{n=0}^\infty\pi_n,\qquad\pi_n:=\sum_{J\subset\N:|J|=n}\Big(\prod_{j\in J}D_j^\circ\Big)\Big(\prod_{j\in \N\setminus J}\E_j\Big),\]
where $\pi_n$ is a well-defined projector on $\Ld^2(\Omega^\circ)$ for each $n$, with $\pi_n\pi_m=0$ for $n\ne m$. We define the $n$-th Glauber chaos as the image $\Hc_n=\pi_n\Ld^2(\Omega^\circ)$, and the above observation leads to the following direct sum decomposition,
\[\Ld^2(\Omega^\circ)\,=\,\bigoplus_{n=0}^\infty\Hc_n.\]
Note that $\pi_0=\E$ and $\Hc_0=\R$.
Since the commutator relations~\eqref{eq:commut-D0} yield $\Lc^\circ\pi_n=n\pi_n$ for all $n$, we deduce $\Lc^\circ=\sum_{n=1}^\infty n\pi_n$, with pseudo-inverse $\Tc^\circ=\sum_{n=1}^\infty\frac1n\pi_n$, and the conclusions~(i)--(iv) easily follow.

\medskip\noindent
We turn to item~(v).
By duality it suffices to argue for $p\ge2$, and by interpolation it suffices to argue for $p=2k$ with $k\in\N$.
Decompose
\begin{eqnarray*}
\|\Tc^\circ Y\|_{\Ld^{2k}(\Omega^\circ)}^{2k}&=&\sum_{n_1,\ldots,n_{2k}\ge1}\frac1{n_1\ldots n_{2k}}\expec{(\pi_{n_1} Y)\ldots(\pi_{n_{2k}} Y)}\\
&\le&(2k)!\sum_{n_1\ge\ldots\ge n_{2k}\ge1}\frac1{n_1\ldots n_{2k}}\big|\expec{(\pi_{n_1} Y)\ldots(\pi_{n_{2k}} Y)}\!\big|,
\end{eqnarray*}
and note that for $n_1\ge\ldots\ge n_{2k}\ge1$ we can write
\[\expec{(\pi_{n_1} Y)\ldots(\pi_{n_{2k}} Y)}=\expec{(\pi_{n_1} Y)\,\pi_{n_1}\big((\pi_{n_2}Y)\ldots(\pi_{n_{2k}} Y)\big)},\]
where by definition of the $\pi_n$'s the expression $\pi_{n_1}((\pi_{n_2}Y)\ldots(\pi_{n_{2k}} Y))$ vanishes whenever there holds $n_{1}>\sum_{j=2}^{2k}n_j$. Restricting the sum to $n_{1}\le\sum_{j=2}^{2k}n_j$, noting that this constraint implies $n_1\le 2kn_2$, and appealing to Jensen's inequality, we find
\begin{eqnarray*}
\|\Tc^\circ Y\|_{\Ld^{2k}(\Omega^\circ)}^{2k}
&\le&(2k)!\,2k\,\|Y\|_{\Ld^{2k}(\Omega^\circ)}^{2k}\sum_{n_1\ge n_3\ge\ldots\ge n_{2k}\ge1}\frac1{n_1^2}\frac1{n_3\ldots n_{2k}},
\end{eqnarray*}
and the conclusion~(v) follows from a direct computation.
\end{proof}

\subsection{Proof of Theorem~\ref{th:cum-gen}}\label{sec:cumulants}
As a direct application of the above calculus, we prove the following representation formula for the covariance, which is an iid version of the so-called Helffer--Sjöstrand representation formula~\cite{HS-94,Sjostrand-96} (see also~\cite[Lemma~5.1]{DGO1}).
Note that the Poincaré inequality of Lemma~\ref{prop:ES} follows as a consequence.

\begin{lem}[Helffer--Sjöstrand representation formula]\label{lem:HS}
For all $Y,Y'\in\Ld^2(\Omega^\circ)$ there holds
\[\cov{Y}{Y'}=\sum_j\expec{(D^\circ_jY)\Tc^\circ(D^\circ_jY')}.\qedhere\]
\end{lem}

\begin{proof}
By density (cf.~Lemma~\ref{lem:OU}(ii)), it suffices to prove this formula for all $Y\in\Img \Lc^\circ$. Writing $Y=\Lc^\circ U$ for some $U\in\Ld^2(\Omega^\circ)/\R$ in the domain of $\Lc^\circ$, we decompose
\[\cov{Y}{Y'}=\expec{(\Lc^\circ U)Y'}=\sum_j\expec{(D^\circ_jU)(D^\circ_jY')}=\sum_j\expec{(D^\circ_j\Tc^\circ Y)(D^\circ_jY')}.\]
The commutation relations~\eqref{eq:commut-D0} ensure that $D^\circ_j\Tc^\circ=\Tc^\circ D^\circ_j$ on $\Ld^2(\Omega^\circ)/\R$, and the claim follows.
\end{proof}

Next, we upgrade this result into representation formulas for higher-order cumulants. This constitutes an iid version of a formula first proven by Nourdin and Peccati~\cite{Nourdin-Peccati-10} in the Malliavin setting. Note that the formula takes a more complicated form here due to the nonlocality of the Glauber derivative.

\begin{lem}[Representation formulas for cumulants]\label{lem:rep-cum}
For a random variable $Y\in\Ld^\infty(\Omega^\circ)$, for all $j$ and $n\ge0$, we define $\delta_j^n(Y):={\mathbb E}_j'[(Y-Y_j')^{n}]$, where $Y_j'$ coincides with $Y$ with the variable $z_j^\circ$ replaced by an iid copy $z_j'$ and where $\mathbb E_j'$ denotes integration with respect to this iid copy $z_j'$.
Next, we define the following random variables,
\begin{eqnarray*}
\Gamma_0^n(Y)&:=&Y\mathds1_{n=0}\hspace{2.9cm}\text{for $n\ge0$,}\\
\Gamma_1^n(Y)&:=&\sum_j\delta_j^{n}(Y)\,\Tc^\circ(D_j^\circ Y)\qquad\text{for $n\ge1$},
\end{eqnarray*}
and iteratively for $n\ge m\ge0$,
\[\Gamma_{m+1}^{n+1}(Y)\,:=\,-\Gamma_{m}^{n+1}(Y)+\binom{n+1}{m}\sum_j\delta_j^{n+1-m}(Y)\,\Tc^\circ D_j^\circ\Gamma_{m}^{m}(Y),\]
and we set for short $\Gamma_m(Y):=\Gamma_m^m(Y)$ for $m\ge0$.
With these notations, the following representation formula holds for all $m\ge0$,
\begin{equation*}
\kappa^\circ_{m+1}[Y]=\expec{\Gamma_m(Y)}.
\qedhere
\end{equation*}
\end{lem}

\begin{proof}
We argue by induction. Since the result is trivial for $m=0$ with $\kappa^\circ_1[Y]=\expec{Y}$ and $\Gamma_0(Y)=Y$, we may assume that the result is known for all $m\le m_0-1$, for some $m_0\ge1$, and it remains to deduce that it also holds for $m=m_0$.
For that purpose, we start from the following classical recursion relation on cumulants (e.g.~\cite[Proposition~2.2]{Nourdin-Peccati-10}),
\[\expecm{Y^{m_0+1}}=\sum_{m=0}^{m_0}\binom{m_0}m\kappa^\circ_{m+1}(Y)\,\expecm{Y^{m_0-m}},\]
which we shall use in the form,
\begin{equation}\label{eq:last-cum-rec}
\kappa^\circ_{m_0+1}(Y)\,=\,\expecm{Y^{m_0+1}}-\sum_{m=0}^{m_0-1}\binom{m_0}m\kappa^\circ_{m+1}(Y)\,\expecm{Y^{m_0-m}}.
\end{equation}
In view of the Helffer--Sjöstrand representation formula of Lemma~\ref{lem:HS}, we can write
\begin{eqnarray*}
\expecm{Y^{m_0+1}}&=&\expecm{Y}\expecm{Y^{m_0}}+\covm{Y^{m_0}}{Y}\\
&=&\kappa^\circ_1(Y)\,\expecm{Y^{m_0}}+\sum_j\expecm{(D_j^\circ Y^{m_0})\Tc^\circ(D_j^\circ Y)}.
\end{eqnarray*}
Using the following identity, for $a,b\in\R$,
\[a^{m_0}-b^{m_0}\,=\,a^{m_0}-(b-a+a)^{m_0}=\sum_{m=0}^{m_0-1}(-1)^{m_0-1-m}\binom{m_0}{m}a^m(a-b)^{m_0-m},\]
we can write
\begin{equation}\label{eq:D-prod}
D_j^\circ Y^{m_0}\,=\,\mathbb E_j'\big[Y^{m_0}-(Y_j')^{m_0}\big]\,=\,\sum_{m=0}^{m_0-1}(-1)^{m_0-1-m}\binom{m_0}{m}Y^m\,\delta_j^{m_0-m}(Y),
\end{equation}
with $\mathbb E_j', Y_j',\delta_j^{m}$ defined as in the statement.
The above then becomes
\[\expecm{Y^{m_0+1}}\,=\,\kappa^\circ_1(Y)\,\expecm{Y^{m_0}}+\sum_{m=0}^{m_0-1}(-1)^{m_0-1-m}\binom{m_0}{m}\expecm{Y^m\,\Gamma^{m_0-m}_1(Y)},\]
with $\Gamma_1^{m_0-m}$ defined as in the statement.
Provided $m_0-1\ge1$, isolating the contribution of $m=m_0-1$ in the sum, the induction assumption in form of $\kappa^\circ_{2}(Y)=\expec{\Gamma_1(Y)}$ yields
\begin{multline*}
\expecm{Y^{m_0+1}}\,=\,\kappa^\circ_{1}(Y)\,\expecm{Y^{m_0}}+\binom{m_0}1\,\kappa^\circ_{2}(Y)\,\expecm{Y^{m_0-1}}\\
+\binom{m_0}{m_0-1}\covm{Y^{m_0-1}}{\Gamma_{1}(Y)}+\sum_{m=0}^{m_0-2}(-1)^{m_0-1-m}\binom{m_0}{m}\expecm{Y^m\Gamma^{m_0-m}_1(Y)}.
\end{multline*}
Appealing to the Helffer--Sjöstrand representation formula of Lemma~\ref{lem:HS} and to formula~\eqref{eq:D-prod} to rewrite the covariance term, we find
\begin{multline*}
\binom{m_0}{m_0-1}\covm{Y^{m_0-1}}{\Gamma_{1}(Y)}\,=\,\binom{m_0}{m_0-1}\sum_j\expecm{(D_j^\circ Y^{m_0-1})\Tc^\circ(D_j^\circ\Gamma_{1}(Y))}\\
\,=\,\sum_{m=0}^{m_0-2}(-1)^{m_0-2-m}\binom{m_0}{m_0-1}\binom{m_0-1}{m}\sum_j\expecm{Y^m\,\delta_j^{m_0-1-m}(Y)\Tc^\circ(D_j^\circ\Gamma_{1}(Y))}.
\end{multline*}
Inserting this into the above and recognizing the definition of $\Gamma_2^{m_0-m}$ from the statement, with $\binom{m_0}{m_0-1}\binom{m_0-1}{m}=\binom{m_0-m}1\binom{m_0}{m}$,
\begin{multline*}
\expecm{Y^{m_0+1}}\,=\,\kappa^\circ_{1}(Y)\,\expecm{Y^{m_0}}+\binom{m_0}1\,\kappa^\circ_{2}(Y)\,\expecm{Y^{m_0-1}}\\
+\sum_{m=0}^{m_0-2}(-1)^{m_0-2-m}\binom{m_0}{m}\expecm{Y^m\,\Gamma_2^{m_0-m}(Y)}.
\end{multline*}
Iterating the above computation based on the induction assumption, we are led to
\[\expecm{Y^{m_0+1}}\,=\,\sum_{m=0}^{m_0-1}\binom{m_0}{m}\kappa^\circ_{m+1}(Y)\,\expecm{Y^{m_0-m}}+\expecm{\Gamma_{m_0}(Y)}.\]
Comparing this with~\eqref{eq:last-cum-rec}, the conclusion $\kappa^\circ_{m_0+1}(Y)=\expecm{\Gamma_{m_0}(Y)}$ follows.
\end{proof}

We may now conclude the proof of Theorem~\ref{th:cum-gen} as a consequence of the above representation formulas.

\begin{proof}[Proof of Theorem~\ref{th:cum-gen}]
For all $n,r\ge1$, by definition of $\delta_j^n(Y)={\mathbb E}_j'[(Y-Y_j')^n]$ in the statement of Lemma~\ref{lem:rep-cum} above, using Jensen's inequality and the convexity inequality $|a-b|^{nr}\le2^{nr-1}(|a|^{nr}+|b|^{nr})$, we find
\begin{eqnarray}
\|\delta_j^n(Y)\|_{\Ld^r(\Omega^\circ)}&\le&\expec{{\mathbb E}_j'\big[|Y-Y_j'|^{n}\big]^r}^\frac1r\nonumber\\
&\le&\expec{\,{\mathbb E}_j'\big[|Y-Y_j'|^{nr}\big]\,}^\frac1r\nonumber\\
&=&\expec{\,{\mathbb E}_j'\big[\big|(Y-\expec{Y})-(Y_j'-\expec{Y})\big|^{nr}\big]\,}^\frac1r\nonumber\\
&\le&2^{nr}\expec{|Y-\expec{Y}|^{nr}}^\frac1r\nonumber\\
&=&2^{nr}\|D_j^\circ Y\|_{\Ld^{nr}(\Omega^\circ)}^n.\label{eq:jensen-deltan}
\end{eqnarray}
Using this estimate and recalling that $\Tc^\circ$ is bounded in $\Ld^r(\Omega^\circ)$ for all $1<r<\infty$ (cf.~Lemma~\ref{lem:OU}(v)),
a direct computation yields by induction, for $n\ge m\ge0$ and $r\ge1$,
\[\|\Gamma_{m+1}^{n+1}(Y_N)\|_{\Ld^r(\Omega^\circ)}\,\lesssim_{m,n,r}\,\sum_{k=0}^mN^{k+1}\sum_{a_1,\ldots,a_{n+2}\ge1\atop\sum_ja_j=n+k+2}\prod_{j=1}^{n+2}\|D^{a_j}Y_N\|_{\ell^\infty_{\ne}\big(\Ld^{\frac{r}{a_j}(n+k+2)}(\Omega^\circ)\big)}.\]
Inserting this into the representation formula of Lemma~\ref{lem:rep-cum}, the conclusion follows.
\end{proof}

\subsection{Proof of Theorem~\ref{th:2ndP}}\label{sec:idea-Chat}
Let $Y\in\Ld^2(\Omega^\circ)$ with $\expec{Y}=0$ and $\var{Y}=1$. For $h:\R\to\R$ Lipschitz-continuous, we define its Stein transform $S_h$ as the solution of Stein's equation
\[S_h'(x)-xS_h(x)=h(x)-\expec{h(\Nc)}.\]
Writing
\[\expec{h(Y)}-\expec{h(\Nc)}=\expec{S_h'(Y)-YS_h(Y)},\]
the Helffer--Sjöstrand representation formula of Lemma~\ref{lem:HS} yields
\[\expec{h(Y)}-\expec{h(\Nc)}=\expecM{S_h'(Y)-\sum_j(D_j^\circ S_h(Y))\Tc^\circ (D_j^\circ Y)}.\]
If $h$ has Lipschitz constant $1$, then $S_h'$ is known to be also Lipschitz-continuous with $\|S_h\|_{W^{2,\infty}(\R)}\lesssim1$ (cf.~\cite{Stein-86}), and a Taylor expansion yields
\[\big|D_j^\circ S_h(Y)-S_h'(Y)D_j^\circ Y\big|\,\lesssim\,\delta_j^2(Y),\]
with the notation of Lemma~\ref{lem:rep-cum}. Appealing to~\eqref{eq:jensen-deltan} and to the boundedness of $\Tc^\circ$ in~$\Ld^3(\Omega^\circ)$, we are led to
\begin{equation*}
\expec{h(Y)}-\expec{h(\Nc)}\,\lesssim\,\expecM{\Big|1-\sum_j(D_j^\circ Y)\Tc^\circ (D_j^\circ Y)\Big|}
+\sum_j\expec{|D_j^\circ Y|^3}.
\end{equation*}
Recalling that Lemma~\ref{lem:HS} yields $1=\var{Y}=\sum_j\expec{(D_j^\circ Y)\Tc^\circ (D_j^\circ Y)}$, we can write
\begin{equation*}
\expecM{\Big|1-\sum_j(D_j^\circ Y)\Tc^\circ (D_j^\circ Y)\Big|}\,\le\,\varM{\sum_j(D_j^\circ Y)\Tc^\circ (D_j^\circ Y)}^\frac12.
\end{equation*}
Taking the supremum over functions $h$ with Lipschitz constant $1$, we conclude
\begin{equation*}
\dW{Y}{\Nc}\,\lesssim\,\varM{\sum_j(D_j^\circ Y)\Tc^\circ (D_j^\circ Y)}^\frac12+\sum_j\expec{|D_j^\circ Y|^3},
\end{equation*}
and the statement follows after applying the Poincaré inequality of Lemma~\ref{prop:ES} to the first right-hand side term.
The corresponding estimate for the Kolmogorov distance is more involved; the reader is referred to~\cite[Theorem~4.2]{LRP-15}.\qed

\smallskip
\section{Sensitivity estimates for Newton's flow}\label{sec:MF}

In view of the mean-field interaction regime in~\eqref{eq:part-dyn}, a $O(1)$ modification of initial data for a given particle is expected to have only a $O(\frac1N)$ effect on the trajectory of other particles.
Such a sensitivity estimate can be made precise as follows in terms of Glauber calculus by means of a Grönwall argument.

\begin{prop}\label{prop:sensitivity}
Let $\mu_N$ denote the empirical measure~\eqref{eq:emp-meas}
associated with the particle dynamics~\eqref{eq:part-dyn}. For all $m\ge0$, $p\ge1$, and $\phi\in C^\infty_b(\D)$, there holds
\[\Big\|(D^\circ)^m\int_{\D}\phi\,d\mu_N^t\Big\|_{\ell^\infty_{\ne}(\Ld^p(\Omega^\circ))}\,\le\,\frac1{N^{m}}C_me^{C_mt}\|\phi\|_{W^{m,\infty}(\D)}\Big(\int_\D|z|^{mp}dF^\circ(z)\Big)^\frac1p,\]
where we recall the notation~\eqref{eq:not-linft-Dp} for the norm $\ell^\infty_{\ne}$ of iterated Glauber derivatives, and where the constant $C_m$ further depends on $m$ and $\|\nabla V\|_{W^{m,\infty}(\T^d)}$.
\end{prop}

\begin{proof}
We split the proof into two steps, starting with the estimation of iterated Glauber derivatives of particle trajectories.

\medskip
\step1 Bounds on Glauber derivatives of particle trajectories.\\
For all $J\subset[N]$, we write $D_J^\circ=\prod_{j\in J}D_j^\circ$ and we introduce the following short-hand notation, for $p\ge1$,
\begin{gather*}
\mathcal X_{J,p}^t\,:=\,\max_{j}\|D_J^\circ x_{j,N}^t\|_{\Ld^p(\Omega^\circ)},\qquad\mathcal V_{J,p}^t\,:=\,\max_{j}\|D_J^\circ v_{j,N}^t\|_{\Ld^p(\Omega^\circ)},\\
\hat{\mathcal X}_{J,p}^t\,:=\,\max_{j\notin J}\|D_J^\circ x_{j,N}^t\|_{\Ld^p(\Omega^\circ)},\qquad\hat{\mathcal V}_{J,p}^t\,:=\,\max_{j\notin J}\|D_J^\circ v_{j,N}^t\|_{\Ld^p(\Omega^\circ)}.
\end{gather*}
In this first step, we prove for all nonempty subsets $J\subset[N]$ with $|J|=m$,
\begin{eqnarray}
\mathcal X_{J,p}^t+\mathcal V_{J,p}^t&\le&\frac1{N^{m-1}}C_me^{C_mt}\prod_{j\in J}\|D_j^\circ z_{j}^\circ\|_{\Ld^{mp}(\Omega^\circ)},\label{eq:bnd-DX}\\
\hat{\mathcal X}_{J,p}^t+\hat{\mathcal V}_{J,p}^t&\le&\frac1{N^{m}}C_me^{C_mt}\prod_{j\in J}\|D_j^\circ z_{j}^\circ\|_{\Ld^{mp}(\Omega^\circ)},\label{eq:bnd-DhatX}
\end{eqnarray}
where the constant $C_m$ further depends on $m$ and $\|\nabla V\|_{W^{m,\infty}(\T^d)}$.

\medskip\noindent
We argue by induction and start with the proof of~\eqref{eq:bnd-DX} for $|J|=m=1$, say $J=\{a\}$ with $a\in[N]$.
Taking the Glauber derivative of Newton's equations~\eqref{eq:part-dyn}, we find
\begin{equation}\label{eq:dtDX}
\partial_t\mathcal X_{\{a\},p}\,\le\,\max_j\|\partial_tD^\circ_ax_{j,N}^t\|_{\Ld^p(\Omega^\circ)}\,=\,\mathcal V_{\{a\},p},
\end{equation}
and similarly,
\begin{equation}\label{eq:dtDV0}
\partial_t\mathcal V_{\{a\},p}\,\le\,\max_{j,l}\|D_a^\circ\nabla V(x_{j,N}-x_{l,N})\|_{\Ld^p(\Omega^\circ)}.
\end{equation}
Note that for a smooth function $H:\R^d\to\R$ we can write
\begin{multline}\label{eq:discr-chain}
D_a^\circ H(x_{j,N})=H(x_{j,N})-\E_a[H(x_{j,N})]=\mathbb E_a'[H(x_{j,N})-H(x_{j,N}^a)]\\
=\int_0^1\mathbb E_a'\big[(x_{j,N}-x_{j,N}^a)\cdot\nabla H\big(tx_{j,N}+(1-t)x_{j,N}^a\big)\big]\,dt,
\end{multline}
where $\{x_{j,N}^a\}_j$ stands for $\{x_{j,N}\}_j$ with initial data $z_{a}^\circ$ replaced by an iid copy $z_a'$, and where $\mathbb E_a'$ denotes integration with respect to this iid copy $z_a'$ only, and this implies
\begin{equation}\label{eq:discr-chain1}
\|D_a^\circ H(x_{j,N})\|_{\Ld^p(\Omega^\circ)}\,\lesssim\,\|\nabla H\|_{\Ld^\infty(\R^d)}\|D_a^\circ x_{j,N}\|_{\Ld^p(\Omega^\circ)}.
\end{equation}
This allows to rewrite~\eqref{eq:dtDV0} in the form
\begin{equation}\label{eq:dtDX+}
\partial_t\mathcal V_{\{a\},p}\,\lesssim\,\|\nabla^2V\|_{\Ld^\infty(\T^d)}\,\mathcal X_{\{a\},p}.
\end{equation}
We now appeal to a Grönwall-type argument in the following form, for any smooth functions $A,B,K,L:\R^+\to\R^+$,
\begin{equation}\label{eq:Gron}
\partial_tA\le B\quad\text{and}\quad\partial_tB\le KA+L\quad\implies\quad A,B\le(A(0)+tB(0)+\textstyle\int_0^tL)\,e^{\int_0^t(K\vee1)}.
\end{equation}
The conclusion~\eqref{eq:bnd-DX} for $m=1$ then follows from~\eqref{eq:dtDX} and~\eqref{eq:dtDX+} with initial data
\[\mathcal X_{\{a\},p}|_{t=0}=\|D^\circ_ax_a^\circ\|_{\Ld^p(\Omega^\circ)},\qquad \mathcal V_{\{a\},p}|_{t=0}=\|D^\circ_av_a^\circ\|_{\Ld^p(\Omega^\circ)}.\]
Next, we similarly prove~\eqref{eq:bnd-DhatX} for $J=\{a\}$. Taking the Glauber derivative of Newton's equations~\eqref{eq:part-dyn}, we find
\begin{equation*}
\partial_t\hat{\mathcal X}_{\{a\},p}\,\le\,\hat{\mathcal V}_{\{a\},p},
\end{equation*}
and, distinguishing between the cases $l\ne a$ and $l= a$ in~\eqref{eq:part-dyn},
\begin{equation*}
\partial_t\hat{\mathcal V}_{\{a\},p}\,\le\,\frac{N-1}N\max_{j,l\ne a}\|D_a^\circ\nabla V(x_{j,N}-x_{l,N})\|_{\Ld^p(\Omega^\circ)}
+\frac1N\max_{j\ne a}\|D_a^\circ\nabla V(x_{j,N}-x_{a,N})\|_{\Ld^p(\Omega^\circ)},
\end{equation*}
that is, in view of~\eqref{eq:discr-chain1},
\begin{equation*}
\partial_t\hat{\mathcal V}_{\{a\},p}\,\lesssim\, \|\nabla^2V\|_{\Ld^\infty(\T^d)}\Big(\hat{\mathcal X}_{\{a\},p}
+\frac1N\mathcal X_{\{a\},p}\Big).
\end{equation*}
Combining this with the bound on $\mathcal X_{\{a\},p}$, the conclusion~\eqref{eq:bnd-DhatX} for $m=1$ follows from Grönwall's inequality~\eqref{eq:Gron} with initial data
\[\hat{\mathcal X}_{\{a\},p}|_{t=0}=\hat{\mathcal V}_{\{a\},p}|_{t=0}=0.\]

\medskip\noindent
Now that~\eqref{eq:bnd-DX} and~\eqref{eq:bnd-DhatX} are proven for $m=1$, we assume that they are known to hold for all $m\le m_0-1$, for some $m_0\ge2$, and we then show that they must also hold for $m=m_0$. Let $J\subset[N]$ with $|J|=m_0$.
Taking the iterated Glauber derivative of Newton's equations~\eqref{eq:part-dyn}, we find
\begin{equation*}
\partial_t\mathcal X_{J,p}\,\le\,\mathcal V_{J,p},
\end{equation*}
and
\begin{equation}\label{eq:dtDV0-iter}
\partial_t\mathcal V_{J,p}\,\le\,\max_{j,l}\|D_J^\circ\nabla V(x_{j,N}-x_{l,N})\|_{\Ld^p(\Omega^\circ)}.
\end{equation}
Iterating the chain rule~\eqref{eq:discr-chain} in form of Faà di Bruno's formula, we obtain the following higher-order version of~\eqref{eq:discr-chain1}, for any smooth function $H:\R^d\to\R$,
\begin{equation}\label{eq:discr-chain2}
\|D_J^\circ H(x_{j,N})\|_{\Ld^p(\Omega^\circ)}\,\lesssim_{m_0}\,\|H\|_{W^{m_0,\infty}(\R^d)}\sum_{\pi\,\vdash J}\prod_{B\in\pi}\|D_B^\circ x_{j,N}\|_{\Ld^{p\frac{m_0}{|B|}}(\Omega^\circ)}.
\end{equation}
Inserting this formula into~\eqref{eq:dtDV0-iter}, we are led to
\begin{equation}
\partial_t\mathcal V_{J,p}\,\lesssim_{m_0}\,\sum_{\pi\,\vdash J}\sum_{B\in\pi}\mathcal X_{B,p\frac{m_0}{|B|}}\prod_{B'\in\pi\setminus B}\hat{\mathcal X}_{B',p\frac{m_0}{|B'|}},
\end{equation}
where the multiplicative constant further depends on $m_0$ and $\|\nabla V\|_{W^{m_0,\infty}(\R^d)}$.
In view of the induction assumption, this takes the form
\begin{equation*}
\partial_t\mathcal V_{J,p}\,\lesssim_{m_0}\,\mathcal X_{J,p}+\frac1{N^{m_0-1}}e^{C_{m_0}t}\prod_{j\in J}\|D_j^\circ z_{j}^\circ\|_{\Ld^{m_0p}(\Omega^\circ)},
\end{equation*}
and the conclusion~\eqref{eq:bnd-DX} with $m=m_0$ follows from Grönwall's inequality~\eqref{eq:Gron} with initial data $\mathcal X_{J,p}|_{t=0}=\mathcal V_{J,p}|_{t=0}=0$. The corresponding proof of~\eqref{eq:bnd-DhatX} is similar.

\medskip
\step2 Conclusion.\\
Distinguishing the contribution of particles with index inside or outside $J$ to the empirical measure $\mu_N$, we can write
\begin{eqnarray*}
\Big\|D^\circ_J\int_\D\phi\,d\mu_N\Big\|_{\Ld^p(\Omega^\circ)}&\le&\frac1N\sum_{j=1}^N\|D_J^\circ\phi(z_{j,N})\|_{\Ld^p(\Omega^\circ)}\\
&\le&\max_{j\notin J}\|D^\circ_J\phi(z_{j,N})\|_{\Ld^p(\Omega^\circ)}+\frac{|J|}N\max_j\|D^\circ_J\phi(z_{j,N})\|_{\Ld^p(\Omega^\circ)},
\end{eqnarray*}
hence, for $|J|=m$, in view of the chain rule~\eqref{eq:discr-chain2} and of the results~\eqref{eq:bnd-DX}--\eqref{eq:bnd-DhatX} of Step~1,
\[\Big\|D^\circ_J\int_{\D}\phi\,d\mu_N^t\Big\|_{\Ld^p(\Omega^\circ)}\,\le\, \frac1{N^{m}}C_{m}e^{C_mt}\|\phi\|_{W^{m,\infty}(\D)}\prod_{j\in J}\|D_j^\circ z_{j}^\circ\|_{\Ld^{mp}(\Omega^\circ)},\]
where the constant $C_m$ further depends on $m$ and $\|\nabla V\|_{W^{m,\infty}(\T^d)}$.
Noting that the product $\prod_{j\in J}\|D_j^\circ z_{j}^\circ\|_{\Ld^{mp}(\Omega^\circ)}$ is bounded by $(\int_\D|z|^{mp}dF^\circ(z))^{1/p}$, the conclusion follows.
\end{proof}

\smallskip
\section{Optimal a priori estimates on correlations}\label{sec:th1}
This section is devoted to the proof of Theorem~\ref{th:cumulant}.
In view of the higher-order Poincaré inequality of Theorem~\ref{th:cum-gen} and the sensitivity estimates of Proposition~\ref{prop:sensitivity}, it only remains to draw the link between correlation functions and cumulants of the empirical measure.
Since for all $k\ge1$ the initial $W^{k,1}$ regularity of marginals of $F_N$ is propagated by the Liouville equation~\eqref{eq:Liouville} uniformly in $N$ (cf.~Remark~\ref{rem:cum}(ii)), the stated a priori estimates for $G_N^{m+1}$ in $W^{-2m,1}(\D^{m+1})$ is equivalent to an estimate in $(W^{2m,\infty}(\D^{m+1}))^*$. By linearity, it thus suffices to prove the following.

\begin{prop}\label{prop:cum}
For $0\le m\le N-1$, the $(m+1)$-particle correlation function $G_N^{m+1}$ is of order $O(N^{-m})$ in the following sense, for all $t\ge0$ and $\phi\in C^\infty_b(\D)$,
\begin{equation}\label{eq:res-th1}
\bigg|\int_{\D^{m+1}}\phi^{\otimes(m+1)}\, G_N^{m+1;\,t}\bigg|\,\le\,\frac1{N^{m}}C_me^{C_mt}\sum_{a_1,\ldots,a_{m+1}\ge1\atop\sum_ja_j=2m}\prod_{j=1}^{m+1}\|\phi\|_{W^{a_j,\infty}(\D)},
\end{equation}
where the constant $C_m$ further depends on $m$, $\|\nabla V\|_{W^{m,\infty}(\T^d)}$, and $\int_\D|z|^{2m}dF^\circ(z)$.
\end{prop}

\begin{proof}
Combining the higher-order Poincaré inequality of Theorem~\ref{th:cum-gen} with the sensitivity estimates of Proposition~\ref{prop:sensitivity}, we obtain the following a priori estimates for cumulants of the empirical measure~\eqref{eq:emp-meas}, for all $m\ge0$ and $\phi\in C^\infty_b(\D)$,
\begin{equation}\label{eq:pre-th1}
\kappa^\circ_{m+1}\Big[\int_\D\phi\, d\mu_N^t\Big]\,\le\,\frac1{N^{m}}C_me^{C_mt}\sum_{a_1,\ldots,a_{m+1}\ge1\atop\sum_ja_j=2m}\prod_{j=1}^{m+1}\|\phi\|_{W^{a_j,\infty}(\D)},
\end{equation}
where the constant $C_m$ further depends on $m$, $\|\nabla V\|_{W^{m,\infty}(\T^d)}$, and $\int_\D|z|^{2m}dF^\circ(z)$.
It remains to massage this estimate and draw the link with correlation functions.
We start from~\eqref{eq:def-kappa-m} in the form
\begin{equation*}
\kappa^\circ_{m+1}\Big[\int_\D\phi\, d\mu_N\Big]\,=\,\sum_{\pi\vdash[m+1]}(-1)^{|\pi|-1}(|\pi|-1)!\prod_{B\in\pi}\expecM{\Big(\int_\D\phi\, d\mu_N\Big)^{|B|}},
\end{equation*}
where the moments of the empirical measure can be computed as follows,
\begin{eqnarray*}
\expecM{\Big(\int_\D\phi\, d\mu_N\Big)^n}&=&\frac1{N^n}\sum_{j_1,\ldots, j_n=1}^N\expecM{\prod_{l=1}^n\phi(z_{j_l,N})}\\
&=&\frac1{N^n}\sum_{\pi\vdash[n]}N(N-1)\ldots(N-|\pi|+1)\int_{\D^{|\pi|}}\Big(\bigotimes_{B\in\pi}\phi^{|B|}\Big)\,F_N^{|\pi|},
\end{eqnarray*}
and where marginals of $F_N$ can be expressed in terms of the cluster expansion~\eqref{eq:cluster},
\[F_N^n(z_{[n]})=\sum_{\pi\vdash[n]}\prod_{B\in\pi}G_N^{|B|}(z_B).\]
Combining these identities, we manage to express cumulants of the empirical measure in terms of correlation functions: after straightforward simplifications, we find
\begin{multline}\label{eq:cum-correl}
\kappa^\circ_{m+1}\Big[\int_\D\phi\, d\mu_N\Big]\\
\,=\,\sum_{\pi\vdash[m+1]}N^{|\pi|-m-1}\sum_{\rho\vdash\pi}K_N(\rho)\int_{\D^{|\pi|}}\Big(\bigotimes_{B\in \pi}\phi^{|B|}\Big)\Big(\bigotimes_{D\in\rho}G_N^{|D|}(z_D)\Big)\,dz_\pi,
\end{multline}
where the coefficients are given by
\[K_N(\rho):=\sum_{\sigma\vdash\rho}(-1)^{|\sigma|-1}(|\sigma|-1)!\,\Big(\prod_{C\in\sigma}\big(1-\tfrac1N\big)\ldots\big(1-\tfrac{(\sum_{D\in C}|D|)-1}N\big)\Big).\]
Isolating $\int_{\D^{m+1}}\phi^{\otimes(m+1)}G_N^{m+1}$ in the right-hand side of~\eqref{eq:cum-correl} (this term is obtained for the choice $\pi=\{\{1\},\ldots,\{m+1\}\}$ and $\rho=\{\pi\}$),
and applying the bound~\eqref{eq:pre-th1} on the left-hand side of~\eqref{eq:cum-correl}, we deduce
\begin{multline*}
\bigg|\int_{\D^{m+1}}\phi^{\otimes(m+1)}G_N^{m+1;t}\bigg|
\,\le\,\frac1{N^{m}}C_me^{C_mt}\sum_{a_1,\ldots,a_{m+1}\ge1\atop\sum_ja_j=2m}\prod_{j=1}^{m+1}\|\phi\|_{W^{a_j,\infty}(\D)}\\
+\sum_{\pi\vdash[m+1]}N^{|\pi|-m-1}\sum_{\rho\vdash\pi}K_N(\rho)\,\mathds1_{\rho\ne\{\{\{1\},\ldots,\{m+1\}\}\}}\int_{\D^{|\pi|}}\Big(\bigotimes_{B\in \pi}\phi^{|B|}\Big)\Big(\bigotimes_{D\in\rho}G_N^{|D|}(z_D)\Big)\,dz_\pi.
\end{multline*}
Since the right-hand side only involves correlation functions $G_N^k$ with $k\le m$,
and noting that $|K_N(\rho)|\lesssim_m N^{1-|\rho|}$,
the conclusion~\eqref{eq:res-th1} follows by induction.
\end{proof}

\smallskip
\section{Mean-field approximation}\label{sec:MFL}
This section is devoted to the proof of Corollary~\ref{cor:MFL}.
In terms of the $2$-particle correlation function $G_N^2=F_N^2-(F_N^1)^{\otimes2}$, the BBGKY equation~\eqref{eq:BBGKY} for the first marginal $F_N^1$ takes the form
\begin{equation*}
\partial_t F_N^{1}+v\cdot\nabla_{x}F_N^{1}=(\nabla V\ast F_N^1)\cdot\nabla_{v}F_N^{1}+E_N,
\end{equation*}
with
\[E_N\,:=\,-\frac{1}N(\nabla V\ast F_N^1)\cdot\nabla_{v}F_N^{1}+\frac{N-1}N\int_\D\nabla V(x-x_*)\cdot\nabla_{v}G_N^{2}(z,z_*)\,dz_*,\]
so that $F_N^1$ satisfies the Vlasov equation~\eqref{eq:Vlasov} up to an error term $E_N$ involving $G_N^2$.
In order to estimate this error term, we argue by duality: given a test function $h\in C^\infty_c(\D)$, integration by parts yields
\begin{multline*}
\bigg|\int_\D hE_N\bigg|\,\le\,\frac1N\bigg|\int_\D\nabla_vh\cdot(\nabla V\ast F_N^1)F_N^1\bigg|\\
+\bigg|\int_{\D^2}\nabla_{v_1}h(z_1)\cdot\nabla V(x_1-x_2)G_N^2(z_1,z_2)dz_1dz_2\bigg|,
\end{multline*}
and thus, recalling $\|F_N^1\|_{\Ld^1(\D)}=1$ and using the a priori estimate of Proposition~\ref{prop:cum} for~$G_N^2$,
\begin{equation*}
\bigg|\int_\D hE_N\bigg|\,\le\,\frac1NCe^{Ct}\|\nabla_{v}\langle\nabla\rangle h\|_{\Ld^\infty(\D)},
\end{equation*}
which entails by duality,
\begin{equation}\label{eq:est-EN}
\|E_N^t\|_{W^{-2,1}(\D)}\,\le\,\frac1NCe^{Ct},
\end{equation}
where the constant $C$ further depends on $\|\nabla V\|_{W^{2,\infty}(\T^d)}$ and $\int_\D|z|^2dF^\circ(z)$.
Next, we appeal to a stability result for the Vlasov equation in $W^{-2,1}(\D)$.
Due to commutator issues, stability is in fact obtained only in $W^{-k,p}(\D)$ with $1<p<\infty$, in the following form.

\begin{lem}\label{lem:stab}
For $1<p<\infty$, given an initial data $F^\circ\in\Pc\cap\Ld^p(\D)$ and a perturbation $E\in\Ld_\loc^\infty(\R^+;\Ld^{p}(\D))$,
let $F_1\in\Ld^\infty_\loc(\R^+;\Pc\cap\Ld^p(\D))$ be a solution of the Vlasov equation~\eqref{eq:Vlasov} and let $F_2\in\Ld^\infty_\loc(\R^+;\Pc\cap\Ld^{p}(\D))$ be a solution of the following perturbed equation,
\begin{equation*}
\partial_tF_2+v\cdot\nabla_xF_2=(\nabla V\ast F_2)\cdot\nabla_vF_2+E,
\end{equation*}
with $F_1|_{t=0}=F_2|_{t=0}=F^\circ$. Further assume that for some $K\ge1$ the solutions $F_1^t$ and $F_2^t$ are compactly supported in $\T^d\times B(0,K\langle t\rangle)$ for all $t\ge0$.
Then the following stability estimate holds, for all $r\ge1+2d\frac{p-1}p$ and $t\ge0$,
\[\|F_1^t-F_2^t\|_{W^{-r,p}(\D)}\,\le C_{p,r}\exp\Big(C_{p,r}\,\langle t\rangle(K\langle t\rangle)^{d\frac{p-1}p}\Big)\sup_{0\le t'\le t}\|E^{t'}\|_{W^{-r,p}(\D)},\]
where the constant $C_{p,r}$ further depends on $p,r,s$, $\|\nabla V\|_{W^{s,\infty}(\T^d)}$, and $\|F^\circ\|_{\Ld^p(\D)}$, for any $s>r$.
In addition, by Hölder's inequality and the compactness assumption, the $W^{-r,p}(\D)$ norm in the left-hand side can be replaced by a $W^{-r,1}(\D)$ norm.
\end{lem}

In order to deduce Corollary~\ref{cor:MFL}, first note that the error estimate~\eqref{eq:est-EN} together with the Sobolev inequality ensures, for~$1<p<\infty$ and $\delta\ge 4d\frac{p-1}{p}$,
\begin{equation}\label{eq:apriori-G2}
\|E_N^t\|_{W^{-2-\delta,p}(\D^2)}\,\lesssim_p\,\|E_N^t\|_{W^{-2,1}(\D^2)}\,\le\,\frac1NCe^{Ct}.
\end{equation}
Second, provided that the initial data $F^\circ$ is supported in $\T^d\times B(0,K)$, Newton's equations~\eqref{eq:part-dyn} ensure that $F_N^t$ is supported in $(\T^d\times B(0,K+t\|\nabla V\|_{\Ld^\infty(\T^d)}))^N$, hence $F_N^{1;t}$ in $\T^d\times B(0,K+t\|\nabla V\|_{\Ld^\infty(\T^d)})$, and similarly the solution~$F^t$ of the Vlasov equation~\eqref{eq:Vlasov} is also supported in $\T^d\times B(0,K+t\|\nabla V\|_{\Ld^\infty(\T^d)})$. Corollary~\ref{cor:MFL} is then a direct consequence of the above stability result for the Vlasov equation. We turn to the proof of the latter.

\begin{proof}[Proof of Lemma~\ref{lem:stab}]
By standard approximation arguments, we may assume that $F^\circ,E$, hence $F_1,F_2$, are smooth, so that the computations below make sense.
The difference $F=F_2-F_1$ satisfies
\[\partial_tF+v\cdot\nabla_xF=(\nabla V\ast {F_1})\cdot\nabla_vF+(\nabla V\ast F)\cdot\nabla_vF_2+E,\]
with vanishing initial data. We then compute
\begin{eqnarray*}
\frac{d}{dt}\|F\|_{W^{-r,p}(\D)}^p&=&\frac{d}{dt}\int_\D|\langle\nabla\rangle^{-r}F|^p\\
&\le&p\,\|F\|_{W^{-r,p}(\D)}^{p-1}\|E\|_{W^{-r,p}(\D)}-p\int_\D|\langle\nabla\rangle^{-r}F|^{p-2}(\langle\nabla\rangle^{-r}F)(\langle\nabla\rangle^{-r}P),
\end{eqnarray*}
in terms of
\[P\,:=\,v\cdot\nabla_xF-(\nabla V\ast {F_1})\cdot\nabla_vF-(\nabla V\ast F)\cdot\nabla_vF_2.\]
Next, we decompose
\begin{equation*}
\langle\nabla\rangle^{-r}P\,=\,Q-\langle\nabla\rangle^{-r}R+v\cdot\nabla_x\langle\nabla\rangle^{-r}F-(\nabla V\ast {F_1})\cdot\nabla_v\langle\nabla\rangle^{-r}F,
\end{equation*}
in terms of the commutator
\begin{equation*}
Q\,:=\,\nabla_x\cdot[\langle\nabla\rangle^{-r},v]F-\nabla_v\cdot\big[\langle\nabla\rangle^{-r},(\nabla V\ast {F_1})\big]F,
\end{equation*}
and the remainder
\[R:=(\nabla V\ast F)\cdot\nabla_vF_2.\]
Inserting this identity into the above,
we find after straightforward simplifications,
\begin{equation}\label{eq:pre-concl-commut}
\frac{d}{dt}\|F\|_{W^{-r,p}(\D)}
\,\le\,\|E\|_{W^{-r,p}(\D)}
+\|Q\|_{\Ld^{p}(\D)}+\|R\|_{W^{-r,p}(\D)}
\end{equation}
It remains to estimate the commutator~$Q$ and the remainder $R$, and we start with the former.
For that purpose, we appeal to a variant of the Kato-Ponce inequality~\cite{Kato-Ponce-88} as provided by Lemma~\ref{lem:KP} below, to the effect of
\begin{eqnarray*}
\|Q\|_{\Ld^p(\D)}&\lesssim_{p,r,s}&\big(1+\|\nabla V\ast{F_1}\|_{W^{s,\infty}(\T^d)}\big)\|F\|_{W^{-r,p}(\D)}\\
&\le&\big(1+\|\nabla V\|_{W^{s,\infty}(\T^d)}\big)\|F\|_{W^{-r,p}(\D)},
\end{eqnarray*}
for $s>r\ge1$, where the last inequality follows from the assumption that $F_1$ is a probability measure.
We turn to the remainder term
\[R=(\nabla V\ast F)\cdot\nabla_vF_2=(\nabla V\ast F)\cdot\nabla_vF+(\nabla V\ast F)\cdot\nabla_vF_1,\]
which is estimated as follows, for $s>r$,
\[\|R\|_{W^{-r,p}(\D)}\,\lesssim\,\|\nabla V\|_{W^{s,\infty}(\T^d)}\big(\|F\|_{\Ld^1(\D)}\|F\|_{W^{-r,p}(\D)}+\|F\|_{W^{-r,1}(\D)}\|F_1\|_{\Ld^p(\D)}\big).\]
Using Hölder's inequality and the compact support assumption in the form
\[\|F\|_{W^{-r,1}(\D)}\,\lesssim\,(K\langle t\rangle)^{d\frac{p-1}p}\|F\|_{W^{-r,p}(\D)},\]
and recalling that $\|F\|_{\Ld^1(\D)}\le\|F_1\|_{\Ld^1(\D)}+\|F_2\|_{\Ld^1(\D)}=2$,
we are led to
\begin{multline*}
\frac{d}{dt}\|F\|_{W^{-r,p}(\D)}
\,\lesssim_{p,r,s}\,\|E\|_{W^{-r,p}(\D)}\\
+\Big(1+(K\langle t\rangle)^{d\frac{p-1}p}(1+\|F_1\|_{\Ld^p(\D)})\|\nabla V\|_{W^{s,\infty}(\T^d)}\Big)\|F\|_{W^{-r,p}(\D)}.
\end{multline*}
Since the $\Ld^p(\D)$ norm of the solution $F_1$ of the Vlasov equation is conserved,
the conclusion follows from a Grönwall argument.
\end{proof}

We now establish the following variant of the Kato-Ponce commutator inequality~\cite{Kato-Ponce-88} for negative regularity.

\begin{lem}\label{lem:KP}
For $1<p<\infty$ and $s>r\ge1$,
there holds for all $f,g\in C^\infty_c(\R^d)$,
\begin{equation*}
\|\langle\nabla\rangle[\langle\nabla\rangle^{-r},g]f\|_{\Ld^p(\R^d)}\,\lesssim_{p,r,s}\,\|f\|_{W^{-r,p}(\R^d)}\|\nabla g\|_{W^{s-1,\infty}(\R^d)}.\qedhere
\end{equation*}
\end{lem}

\begin{proof}
We argue by duality.
Given $f,g,h\in C^\infty_c(\R^d)$, for $p'=\frac{p}{p-1}$, we can write
\begin{eqnarray*}
\int_{\R^d}h\langle\nabla\rangle[\langle\nabla\rangle^{-r},g]f&=&\int_{\R^d}(\langle\nabla\rangle^{1-r}h)[g,\langle\nabla\rangle^{r}](\langle\nabla\rangle^{-r}f)\\
&\le&\|\langle\nabla\rangle^{-r}f\|_{\Ld^p(\R^d)}\|[\langle\nabla\rangle^{r},g](\langle\nabla\rangle^{1-r}h)\|_{\Ld^{p'}(\R^d)}.
\end{eqnarray*}
We now appeal to the Kato-Ponce inequality~\cite[Lemma~X.1]{Kato-Ponce-88} in the following modified form: for all $v,w\in C^\infty_c(\R^d)$, $1<q<\infty$, and $a\ge1$,
\[\|[\langle\nabla\rangle^a,v]w\|_{\Ld^q(\R^d)}\,\lesssim_{q,a}\,\|\langle\nabla\rangle^{a-1}w\|_{\Ld^q(\R^d)}\|\nabla v\|_{\Ld^\infty(\R^d)}+\|w\|_{\Ld^q(\R^d)}\|\nabla\langle\nabla\rangle^{a-1} v\|_{\Ld^\infty(\R^d)},\]
hence for $b>a\ge1$,
\[\|[\langle\nabla\rangle^a,v]w\|_{\Ld^q(\R^d)}\,\lesssim_{q,a,b}\,\|\langle\nabla\rangle^{a-1}w\|_{\Ld^q(\R^d)}\|\nabla v\|_{W^{b-1,\infty}(\R^d)}.\]
The above then becomes for $1<p<\infty$ and $s>r\ge1$,
\begin{equation*}
\int_{\R^d}h\langle\nabla\rangle[\langle\nabla\rangle^{-r},g]f\,\lesssim_{p,r,s}\,\|h\|_{\Ld^{p'}(\R^d)}\|\langle\nabla\rangle^{-r}f\|_{\Ld^p(\R^d)}\|\nabla g\|_{W^{s-1,\infty}(\R^d)}.
\end{equation*}
Taking the supremum over $h$, the conclusion follows by duality.
\end{proof}

\smallskip
\section{Bogolyubov corrections to mean field}\label{sec:Bogo}
This section is devoted to the proof of Corollary~\ref{cor:Bogo}.
The BBGKY hierarchy~\eqref{eq:BBGKY} can alternatively be written as a hierarchy of equations for correlation functions, and we wish to appeal to Proposition~\ref{prop:cum} to truncate it to the desired accuracy.
We start with a description of the leading-order contribution of $2$-particle correlations to the $1$-particle density.
The conclusion of Corollary~\ref{cor:Bogo} is postponed to the end of this section.

\begin{prop}\label{lem:Bogo}
The evolution of the $1$-particle density is given to next order by
\begin{multline}\label{eq:appr-F1+}
\Big\|\partial_t F_N^{1}+v\cdot\nabla_{x}F_N^{1}-\frac{N-1}N(\nabla V\ast F_N^1)\cdot\nabla_{v}F_N^{1}\\
-\frac1N\int_\D\nabla V(x-x_*)\cdot\nabla_{v}(NG_N^{2})(z,z_*)\,dz_*\Big\|_{W^{-2,1}(\D)}
\,\le\, \frac1{N^2}Ce^{Ct},
\end{multline}
where the constant $C$ further depends on $\|\nabla V\|_{W^{2,\infty}(\T^d)}$ and $\int_\D|z|^2dF^\circ(z)$,
while the contribution of $2$-particle correlations is characterized to leading order as follows, for all $\delta>0$,
\begin{equation}\label{eq:appr-F2}
\big\|\langle\nabla_{z_1}\rangle^{-1}\langle\nabla_{z_2}\rangle^{-1}\big(NG_N^2-H^2\big)\big\|_{W^{-2-\delta,1}(\D^2)}\,\le\, \frac1NC_\delta e^{C_\delta t^{1+\delta}},
\end{equation}
where $H^2$ is defined as in the statement of Corollary~\ref{cor:Bogo} (cf.~\eqref{eq:H-defin}),
and where the constant $C_\delta$ further depends on $\delta$, $\|\nabla V\|_{W^{3+\delta,\infty}(\T^d)}$, $\|F^\circ\|_{\Ld^{1+\delta}(\D)}$, and $\supp F^\circ$.
\end{prop}

\begin{proof}
We start with the proof of~\eqref{eq:appr-F1+}.
The BBGKY equation~\eqref{eq:BBGKY} for $F_N^1$ takes the form
\begin{equation*}
\partial_t F_N^{1}+v\cdot\nabla_{x}F_N^{1}=\frac{N-1}N(\nabla V\ast F_N^1)\cdot\nabla_{v}F_N^{1}+\int_\D\nabla V(x-x_*)\cdot\nabla_{v}G_N^{2}(z,z_*)\,dz_*+P_N,
\end{equation*}
with
\[P_N\,:=\,-\frac{1}N\int_\D\nabla V(x-x_*)\cdot\nabla_{v}G_N^{2}(z,z_*)\,dz_*,\]
the claim~\eqref{eq:appr-F1+} follows from the a priori estimate of Proposition~\ref{prop:cum} in the form
\[\|P_N^t\|_{W^{-2,1}(\D)}\,\le\,\frac1{N^2}Ce^{Ct},\]
where the constant $C$ further depends on $\|\nabla V\|_{W^{2,\infty}(\T^d)}$ and $\int_\D|z|^2dF^\circ(z)$.

\medskip\noindent
Next, we turn to the proof of~\eqref{eq:appr-F2}.
Starting from the BBGKY hierarchy~\eqref{eq:BBGKY} and suitably regrouping the terms,
we obtain the following equation for the second cumulant $G_N^2$,
\begin{equation}\label{eq:G2-reorg}
\partial_tG_N^2+iL_FG_N^2=\frac1NS+E_N,
\end{equation}
where $iL_F$ is the linearized Vlasov operator at $F$, cf.~\eqref{eq:LFH},
where the source term $S$ is given by
\begin{multline*}
S\,:=\,\nabla V(x_1-x_2)\cdot(\nabla_{v_1}-\nabla_{v_2})(F\otimes F)\\
-\big(\nabla V\ast F(x_1)\cdot\nabla_{v_1}+\nabla V\ast F(x_2)\cdot\nabla_{v_2}\big)(F\otimes F),
\end{multline*}
and where the error term $E_N$ takes the form
\[E_N(z_1,z_2)=\widetilde E_N(z_1,z_2)+\widetilde E_N(z_2,z_1),\]
in terms of
\begin{multline*}
\widetilde E_N(z_1,z_2)\,:=\,\frac1N\nabla V(x_1-x_2)\cdot\nabla_{v_1}\Big((F_N^1\otimes F_N^1-F\otimes F)+G_N^2\Big)\\
-\frac1N\Big(\nabla V\ast F_N^1(x_1)\cdot\nabla_{v_1}F_N^1(z_1)F_N^1(z_2)-\nabla V\ast F(x_1)\cdot\nabla_{v_1}F(z_1)F(z_2)\Big)\\
-\frac1N F_N^1(z_2)\int_\D\nabla V(x_1-x_*)\cdot\nabla_{v_1}G_N^2(z_1,z_*)\,dz_*\\
+\nabla_{v_1}\Big(\tfrac{N-2}NF_N^1-F\Big)(z_1)\cdot\int_\D\nabla V(x_1-x_*)G_N^2(z_2,z_*)\,dz_*\\
+\frac{N-2}N \nabla V\ast (F_N^1-F)(x_1)\cdot\nabla_{v_1}G_N^2(z_1,z_2)\\
+\frac{N-2}N \int_\D\nabla V(x_1-x_*)\cdot\nabla_{v_1}G_N^3(z_1,z_2,z_*)\,dz_*.
\end{multline*}
Comparing~\eqref{eq:G2-reorg} with the equation~\eqref{eq:H-defin} for $H^2$ in the statement of Corollary~\ref{cor:Bogo}, we find
\begin{equation}\label{eq:bound-err-GH}
(\partial_t+iL_F)(NG_N^2-H^2)=NE_N,
\end{equation}
with vanishing initial data.
In order to bound $NG_N^2-H^2$, we start by estimating the error term $E_N$.
For that purpose, we argue by duality: given a test function $h\in C^\infty_c(\D^2)$, integrating by parts, recalling that $\|F_N^1\|_{\Ld^1(\D)}=1$, and appealing to the a priori estimates of Proposition~\ref{prop:cum} for~$G_N^2$ and~$G_N^3$, we find for all $\delta>0$,
\begin{equation*}
\bigg|\int_{\D^2}h E_N^t\bigg|
\,\lesssim\,\frac1NC_\delta e^{C_\delta t}\Big(\frac1N+\|F_N^{1;t}-F^t\|_{W^{-2-\frac\delta2,1}(\D)}\Big)\|\langle\nabla_{z_1}\rangle\langle\nabla_{z_2}\rangle\langle\nabla_{(z_1,z_2)}\rangle^{2+\delta}h\|_{\Ld^\infty(\D^2)},
\end{equation*}
and thus, by duality,
setting for shortness $A_{2+\delta}:=\langle\nabla_{z_1}\rangle^{-1}\langle\nabla_{z_2}\rangle^{-1}\langle\nabla_{(z_1,z_2)}\rangle^{-2-\delta}$,
\[\|A_{2+\delta}E_N^t\|_{\Ld^{1}(\D^2)}\,\le\,\frac1NC_\delta e^{C_\delta t}\Big(\frac1N+\|F_N^{1;t}-F^t\|_{W^{-2-\frac\delta2,1}(\D)}\Big),\]
where the constant $C_\delta$ further depends on $\delta$, $\|\nabla V\|_{W^{3+\delta,\infty}(\T^d)}$, and $\int_\D|z|^4dF^\circ(z)$.
In view of Corollary~\ref{cor:MFL}, we deduce for all $\delta>0$,
\begin{equation*}
\|A_{2+\delta}E_N^t\|_{\Ld^{1}(\D^2)}\,\le\,\frac1{N^2}C_\delta e^{C_\delta t^{1+\delta}},
\end{equation*}
where the constant $C_\delta$ further depends on $\supp F^\circ$.
Now using the stability result of Lemma~\ref{lem:stab-G2} below for equation~\eqref{eq:bound-err-GH}, together with the Sobolev inequality as in~\eqref{eq:apriori-G2}, the conclusion~\eqref{eq:appr-F2} follows.
\end{proof}

\begin{lem}\label{lem:stab-G2}
For $1<p<\infty$, given an initial data $F^\circ\in\Pc\cap\Ld^p(\D)$ and a perturbation $E\in\Ld_\loc^\infty(\R^+;\Ld^{p}(\D^2))$,
let $F\in\Ld^\infty_\loc(\R^+;\Pc\cap\Ld^p(\D))$ be a solution of the Vlasov equation~\eqref{eq:Vlasov}
and let $G\in\Ld^\infty_\loc(\R^+;\Ld^{p}(\D^2))$ be a solution of the following linear equation,
\begin{equation*}
\partial_tG+iL_FG=E,
\end{equation*}
with $F|_{t=0}=F^\circ$ and $G|_{t=0}=0$. Further assume that for some $K\ge1$ the solutions $F^t$ and $G^t$ are compactly supported in $\T^d\times B(0,K\langle t\rangle)$ and in $(\T^d\times B(0,K\langle t\rangle))^2$, respectively, for all $t\ge0$.
Then, for all $r\ge0$, setting $A_r:=\langle\nabla_{z_1}\rangle^{-1}\langle\nabla_{z_2}\rangle^{-1}\langle\nabla_{(z_1,z_2)}\rangle^{-r}$, the following stability estimate holds,
\[\|A_rG^t\|_{\Ld^p(\D^2)}\le C_{p,r}\exp\Big(C_{p,r}\langle t\rangle(K\langle t\rangle)^{d\frac{p-1}p}\Big)\sup_{0\le t'\le t}\|A_rE^{t'}\|_{\Ld^p(\D^2)},\]
where the constant $C_{p,r}$ further depends on $p,r,s$, $\|\nabla V\|_{W^{s+1,\infty}(\T^d)}$, and $\|F^\circ\|_{\Ld^p(\D)}$ for any $s>r$.
In addition, by Hölder's inequality and the compact support assumption, the $\Ld^{p}(\D^2)$ norm in the left-hand side can be replaced by an $\Ld^{1}(\D^2)$ norm.
\end{lem}

\begin{proof}
By standard approximation arguments, we may assume that $F^\circ,E$, hence $F,G$, are smooth, so that the computations below make sense.
We compute
\begin{equation*}
\frac{d}{dt}\|A_rG\|_{\Ld^{p}(\D^2)}^p
\,\le\,p\,\|A_rG\|_{\Ld^{p}(\D^2)}^{p-1}\|A_rE\|_{\Ld^{p}(\D^2)}
-p\int_{\D^2}|A_rG|^{p-2}(A_rG)(A_riL_FG).
\end{equation*}
By definition of $iL_F$, we may decompose
\begin{multline*}
A_riL_FG\,=\,Q_r-A_rR+\big(v_1\cdot\nabla_{x_1}+v_2\cdot\nabla_{x_2}\big)A_rG\\
-\big(\nabla V\ast F(x_1)\cdot\nabla_{v_1}+\nabla V\ast F(x_2)\cdot\nabla_{v_2}\big)A_rG,
\end{multline*}
in terms of the commutator
\begin{multline*}
Q_r\,:=\,\nabla_{x_1}\cdot[A_r,v_1]G+\nabla_{x_2}\cdot[A_r,v_2\cdot]G\\
-\nabla_{v_1}\cdot\big[A_r,\nabla V\ast F(x_1)\big]G
-\nabla_{v_2}\cdot\big[A_r,\nabla V\ast F(x_2)\big]G,
\end{multline*}
and the remainder
\begin{multline*}
R(z_1,z_2)\,:=\,\nabla_{v_1}F(z_1)\cdot\int_\D\nabla V(x_1-x_*)\,G(z_2,z_*)\,dz_*\\
+\nabla_{v_2}F(z_2)\cdot\int_\D\nabla V(x_2-x_*)\,G(z_1,z_*)\,dz_*.
\end{multline*}
After straightforward simplifications, the above becomes
\begin{equation}\label{eq:pre-gron-AG}
\frac{d}{dt}\|A_rG\|_{\Ld^{p}(\D^2)}
\,\le\,\|A_rE\|_{\Ld^{p}(\D^2)}+\|Q_r\|_{\Ld^{p}(\D^2)}+\|A_rR\|_{\Ld^{p}(\D^2)}.
\end{equation}
It remains to estimate the commutator~$Q_r$ and the remainder $R$, and we start with the former.
For that purpose, we appeal to Lemma~\ref{lem:KP} in the following form, for $s>r\ge0$,
\begin{eqnarray*}
\|Q_r\|_{\Ld^p(\D^2)}&\lesssim_{p,r,s}&\big(1+\|\nabla V\ast F\|_{W^{s+1,\infty}(\T^d)}\big)\|A_rG\|_{\Ld^p(\D^2)}\\
&\le&\big(1+\|\nabla V\|_{W^{s+1,\infty}(\T^d)}\big)\|A_rG\|_{\Ld^p(\D^2)},
\end{eqnarray*}
where the last inequality follows from the assumption that $F$ is a probability measure.
We turn to the remainder term $R$, which is estimated as follows, for $s>r$,
\[\|A_rR^t\|_{\Ld^p(\D^2)}\,\lesssim_{p,r,s}\,(K\langle t\rangle)^{d\frac{p-1}p}\|\nabla V\|_{W^{s+1,\infty}(\T^d)}\|F^t\|_{\Ld^p(\D)}\|A_rG^t\|_{\Ld^p(\D^2)}.\]
Since the $\Ld^p(\D)$ norm of the solution $F$ of the Vlasov equation~\eqref{eq:Vlasov} is conserved, the conclusion follows from a Grönwall argument.
\end{proof}

We finally turn to the proof of Corollary~\ref{cor:Bogo}, which is easily deduced from the estimates of Proposition~\ref{lem:Bogo} above.

\begin{proof}[Proof of Corollary~\ref{cor:Bogo}]
Combining~\eqref{eq:appr-F1+} and~\eqref{eq:appr-F2} in Proposition~\ref{lem:Bogo} yields
\begin{multline*}
\partial_t F_N^{1}+v\cdot\nabla_{x}F_N^{1}=\frac{N-1}N(\nabla V\ast F_N^1)\cdot\nabla_{v}F_N^{1}\\
+\frac1N\int_\D\nabla V(x-x_*)\cdot\nabla_{v}H^{2}(z,z_*)\,dz_*+R_N,
\end{multline*}
with
\[\|R_N^t\|_{W^{-4-\delta,1}(\D)}
\,\le\, \frac1{N^2}C_\delta e^{C_\delta t^{1+\delta}},\]
where the constant $C_\delta$ further depends on $\delta$, $\|\nabla V\|_{W^{4+\delta,\infty}(\T^d)}$, $\|F^\circ\|_{\Ld^{1+\delta}(\D)}$, and $\supp F^\circ$.
Comparing with the equation~\eqref{eq:HN1-defin} for $H_N^1$, the difference $L_N^1:=F_N^1-H_N^1$ satisfies
\begin{equation*}
\partial_t L_N^{1}+v\cdot\nabla_{x}L_N^{1}
=\frac{N-1}N(\nabla V\ast L_N^1)\cdot\nabla_{v}H_N^{1}+\frac{N-1}N(\nabla V\ast F_N^1)\cdot\nabla_{v}L_N^1+R_N,
\end{equation*}
with vanishing initial data.
Repeating the stability argument in the proof of Lemma~\ref{lem:stab}, we find for all $1<p<\infty$, $s>r\ge1$, and $u\ge0$,
\begin{multline*}
\frac{d}{dt}\|L_N^1\|_{W^{-r,p}}\,\lesssim_{p,r,s}\,\|R_N\|_{W^{-r,p}(\D)}\\
+\Big(1+\|\nabla V\|_{W^{s,\infty}(\T^d)}+(K\langle t\rangle)^{d\frac{p-1}p}\|\nabla V\|_{W^{s+u,\infty}(\T^d)}\|H_N^1\|_{W^{-u,p}(\D)}\Big)\|L_N^1\|_{W^{-r,p}(\D)}.
\end{multline*}
Choosing $u=r$ and estimating for $r\ge2d\frac{p-1}p$,
\[\|H_N^1\|_{W^{-r,p}(\D)}\le \|F^1_N\|_{W^{-r,p}(\D)}+\|L_N^1\|_{W^{-r,p}(\D)}\lesssim1+\|L_N^1\|_{W^{-r,p}(\D)},\]
the conclusion follows from a Grönwall argument.
\end{proof}

\smallskip
\section{Fluctuations of the empirical measure}\label{sec:CLT}
This section is devoted to the proof of a central limit theorem for the empirical measure, cf.~Corollary~\ref{cor:CLT}. More precisely, we establish the following.

\begin{prop}\label{prop:CLT}
Given $\phi\in C^\infty_c(\D)$, we set for abbreviation
\[Y_{N,\phi}^t\,:=\,\int_{\D}\phi\,d\mu_N^t,\qquad
(\sigma_{N,\phi}^t)^2\,:=\,\Var\big[{\sqrt{N}\, Y_{N,\phi}^t}\big],\]
and the reduced random variable
\[X_{N,\phi}^t\,:=\,\frac{1}{\sigma_{N,\phi}^t}\sqrt N \big(Y_{N,\phi}^t-\expecm{Y_{N,\phi}^t}\big).\]
Denote by $F$ the solution of the Vlasov equation~\eqref{eq:Vlasov}, and let $H^2$ be defined as in the statement of Corollary~\ref{cor:Bogo}.
Then, the following properties hold.
\begin{enumerate}[(i)]
\item \emph{Asymptotic normality:} For all $N\ge1$ and $t\ge0$,
\begin{multline}\label{eq:as-norm}
\qquad\dW{X_{N,\phi}^t}{\Nc}+\dK{X_{N,\phi}^t}{\Nc}\\
\,\lesssim\,\frac1{\sqrt{N}}Ce^{Ct}
\bigg(\frac{\|\nabla\phi\|_{\Ld^\infty}^3}{(\sigma_{N,\phi}^t)^{3}}
+\frac{\|\nabla\phi\|_{\Ld^\infty}\|\nabla\phi\|_{W^{1,\infty}}}{(\sigma_{N,\phi}^t)^2}\bigg),
\end{multline}
where the constant $C$ further depends on $\|\nabla V\|_{W^{2,\infty}(\T^d)}$, $\int_\D|z|^8dF^\circ(z)$.
\item \emph{Convergence of the variance:} For all $t\ge0$, there holds $\sigma_{N,\phi}^t\to\sigma_\phi^t$ as $N\uparrow\infty$, where the limit is given by
\[\qquad(\sigma_\phi^t)^2:=\bigg(\int_\D\phi^2F^t-\Big(\int_\D\phi F^t\Big)^2\bigg)+\int_\D(\phi\otimes\phi)\,H^{2;t},\]
and there holds for all $N\ge1$ and $\delta>0$,
\begin{equation}\label{eq:conv-sigma}
\qquad|(\sigma_{N,\phi}^t)^2-(\sigma_\phi^t)^2|\,\le\,\frac1NC_\delta e^{C_\delta t^{1+\delta}}\|\phi\|_{W^{3+\delta,\infty}}^2,
\end{equation}
where the constant $C_\delta$ further depends on $\delta$, $\|\nabla V\|_{W^{3+\delta,\infty}(\T^d)}$, $\|F^\circ\|_{\Ld^{1+\delta}(\D)}$, $\supp F^\circ$.
\item \emph{Alternative characterization of the limiting variance:} For all $t\ge0$, the limit $\sigma_\phi^t$ can be described as
\[(\sigma_\phi^t)^2\,=\,\varr{\int_\D\phi\,U^t[\G^\circ]},\]
where the Gaussian field $\G^\circ$ is the distributional limit in law of $\sqrt N(\mu_N^\circ-F^\circ)$,
and where $U[\G^\circ]$ denotes the solution of the linearized Vlasov equation at $F$ with initial data $\G^\circ$, cf.~\eqref{eq:lin-Vlas-fl}.
\qedhere
\end{enumerate}
\end{prop}

Combining items~(i) and~(ii), together with Corollary~\ref{cor:MFL},
and noting that for a random variable $Y$  there holds for all $a,b>0$ and $z\in\R$,
\begin{eqnarray*}
\dW{Y-z}{a\Nc}&\le&b\dW{\tfrac1bY}{\Nc}+|a-b|+|z|,\\
\dK{Y-z}{a\Nc}&\le&\dK{\tfrac1bY}{\Nc}+\frac{1}{a\wedge b}(|a-b|+|z|),
\end{eqnarray*}
the conclusion~\eqref{eq:CLT-fin} of Corollary~\ref{cor:CLT} easily follows.
We turn to the proof of the above proposition.

\begin{proof}[Proof of Proposition~\ref{prop:CLT}]
We start with~(i).
We apply the second-order Poincaré inequality for approximate normality as stated in Theorem~\ref{th:2ndP}, to the effect of
\begin{multline*}
\dW{X_{N,\phi}^t}{\Nc}+\dK{X_{N,\phi}^t}{\Nc}\,\lesssim\,\frac1{(\sigma_{N,\phi}^t)^3}N^\frac52\|D^\circ Y_{N,\phi}^t\|_{\ell^\infty_{\ne}(\Ld^6(\Omega^\circ))}^3\\
+\frac1{(\sigma_{N,\phi}^t)^2}\Big(N^\frac32\|D^\circ Y_{N,\phi}^t\|_{\ell^\infty_{\ne}(\Ld^4(\Omega^\circ))}^2+N^\frac52\|D^\circ Y_{N,\phi}^t\|_{\ell^\infty_{\ne}(\Ld^4(\Omega^\circ))}\|(D^\circ)^2 Y_{N,\phi}^t\|_{\ell^\infty_{\ne}(\Ld^4(\Omega^\circ))}\Big).
\end{multline*}
The claim~(i) is a direct consequence of the sensitivity estimates of Proposition~\ref{prop:sensitivity}.

\medskip\noindent
We turn to~(ii).
Computing the variance of the empirical measure in terms of marginals of~$F_N$, we find
\begin{equation*}
(\sigma_{N,\phi})^2=\bigg(\int_\D\phi^2F_N^1-\Big(\int_\D\phi F_N^1\Big)^2\bigg)+\frac{N-1}N\int_\D(\phi\otimes\phi)\,(NG_N^2).
\end{equation*}
(Note that this coincides with identity~\eqref{eq:cum-correl} for $m=1$.)
Hence,
\begin{multline*}
\bigg|(\sigma_{N,\phi})^2-\bigg(\int_\D\phi^2F-\Big(\int_\D\phi F\Big)^2\bigg)-\int_\D(\phi\otimes\phi)\,H^2\bigg|\\
\,\lesssim\,\Big|\int_\D\phi^2(F_N^1-F)\Big|+\|\phi\|_{\Ld^\infty}\Big|\int_\D\phi (F_N^1-F)\Big|+\Big|\int_\D(\phi\otimes\phi)\,\big(NG_N^2-H^2\big)\Big|+\Big|\int_\D(\phi\otimes\phi)\,G_N^2\Big|.
\end{multline*}
The claim~(ii) follows from this in combination with Corollary~\ref{cor:MFL}, with~\eqref{eq:appr-F2} in Proposition~\ref{lem:Bogo}, and with Proposition~\ref{prop:cum}.

\medskip\noindent
It remains to establish~(iii).
By definition, the Gaussian random field $\G^\circ$ has covariance structure given by
\[\varr{\int_\D\phi\, \G^\circ}=\lim_{N\uparrow\infty}\var{\sqrt N\int_\D\phi\,d\mu_N^\circ}=\int_\D\phi^2 F^\circ-\Big(\int_\D\phi F^\circ\Big)^2,\]
that is,
\[h^\circ(z_1,z_2)\,:=\,\covr{\G^\circ(z_1)}{\G^\circ(z_2)}=F^\circ(z_1)\delta(z_1-z_2)-F^\circ(z_1)F^\circ(z_2).\]
In those terms, we can write
\[\varr{\int_\D\phi\,U^t[\G^\circ]}=\int_{\D^2}(\phi\otimes\phi)\,(U^t\otimes U^t)[h^\circ].\]
Defining
\begin{eqnarray*}
h^t(z_1,z_2)&:=&F^t(z_1)\delta(z_1-z_2)-F^t(z_1)F^t(z_2),\\
R^t&:=&(U^t\otimes U^t)[h^\circ]-h^t,
\end{eqnarray*}
this becomes
\begin{equation*}
\varr{\int_\D\phi\,U^t[\G^\circ]}=\bigg(\int_\D\phi^2F^t-\Big(\int_\D\phi F^t\Big)^2\bigg)+\int_{\D^2}(\phi\otimes\phi)\,R^t,
\end{equation*}
while a direct computation shows that the remainder $R$ satisfies
\begin{multline*}
\partial_tR+iL_{F}R=\nabla V(x_1-x_2)\cdot(\nabla_{v_1}-\nabla_{v_2})(F\otimes F)\\
-\big(\nabla V\ast F(x_1)\cdot\nabla_{v_1}+\nabla V\ast F(x_2)\cdot\nabla_{v_2}\big)(F\otimes F),
\end{multline*}
with $R|_{t=0}=0$, where $iL_{F}$ is the linearized Vlasov operator at $F$, cf.~\eqref{eq:LFH}. Hence, we deduce $R=H^2$ by definition~\eqref{eq:H-defin} of $H^2$, and the claim~(iii) follows.
\end{proof}

\smallskip
\section{Lenard--Balescu limit}\label{sec:LB}
This section is devoted to the proof of Corollary~\ref{cor:LB}.
We start with a suitable simplification of the Bogolyubov correction of Corollary~\ref{cor:Bogo} in the spatially homogeneous setting.

\begin{cor}\label{cor:Bogo+}
Let the same assumptions hold as in Theorem~\ref{th:cumulant}, and assume that initial data are spatially homogeneous (that is, $F^\circ(x,v)\equiv f^\circ(v)$).
Let $h_N^1$ satisfy
\begin{equation}\label{eq:HN1-defin0}
\partial_th_N^1\,=\,\frac1N\int_{\T^d}\int_\D\nabla V(x-x_*)\cdot\nabla_vH^2(z,z_*)\,dz_*dx,
\end{equation}
where $H^2$ is the solution of
\begin{equation}\label{eq:H-defin0}
\partial_tH^2+iL^\circ H^2\,=\,S,\qquad S:=\nabla V(x_1-x_2)\cdot(\nabla_{v_1}-\nabla_{v_2})(f^\circ\otimes f^\circ),
\end{equation}
with initial data $h_N^1|_{t=0}=f^\circ$ and $H^2|_{t=0}=0$,
where $iL^\circ$ stands for the $2$-particle linearized Vlasov operator at $f^\circ$,
\begin{multline}\label{eq:LFH0}
iL^\circ H\,:=\,\big(v_1\cdot\nabla_{x_1}+v_2\cdot\nabla_{x_2}\big)H
-\nabla_{v_1}f^\circ (v_1)\cdot\int_\D\nabla V(x_1-x_*)H(z_*,z_2)\,dz_*\\
-\nabla_{v_2}f^\circ(v_2)\cdot\int_\D\nabla V(x_2-x_*)H(z_1,z_*)\,dz_*.
\end{multline}
Then, the $1$-particle velocity density~\eqref{eq:vel-distr} is close to $h_N^1$ in the following sense, for all $t\ge0$ and $\delta>0$,
\[1\wedge\|f_N^{1;t}-h_N^{1;t}\|_{W^{-4-\delta,1}(\R^d)}\,\le\,\frac1{N^2}C_\delta e^{C_\delta t^{1+\delta}},\]
where the constant $C_\delta$ further depends on $\delta$, $\|\nabla V\|_{W^{8+\delta,\infty}(\T^d)}$, $\|F^\circ\|_{\Ld^{1+\delta}(\R^d)}$, $\supp F^\circ$.
\end{cor}

With the above at hand, noting that the error estimate is accurate uniformly in the range $0\le t\le (\log N)^{1-\delta}$, it only remains to compute the long-time limit of the approximate solution $h_N^1$. This is naturally performed by means of Laplace transform, which we use in the following form.

\begin{lem}\label{lem:laplace}
With the notation of Corollary~\ref{cor:Bogo+}, there holds for all $\chi\in C^\infty_c(\R^+)$ and $t_N>0$,
\begin{multline}\label{eq:form-resolv-LB}
\int_0^\infty\chi(\tau)\,(N\partial_th_N^1)|_{t=t_N\tau}\,d\tau\\
\,=\,\int_\R g_\chi(\alpha)\int_{\T^d}\int_\D\nabla V(x-x_*)\cdot\Big(\nabla_v\big(iL^\circ+\tfrac{i\alpha+1}{t_N}\big)^{-1}S\Big)(z,z_*)\,dz_*dx\,d\alpha,
\end{multline}
where $g_\chi(\alpha):=\frac1{2\pi}\int_0^\infty\frac{e^{(i\alpha+1)\tau}}{i\alpha+1}\chi(\tau)\,d\tau$ belongs to $C_b^\infty(\R)$ and satisfies
\[|g_\chi(\alpha)|\lesssim_\chi\langle\alpha\rangle^{-2}\qquad\text{and}\qquad\int_\R g_\chi=\int_0^\infty\chi.\qedhere\]
\end{lem}

\begin{proof}
Corollary~\ref{cor:Bogo+} yields
\[\int_0^\infty\chi(\tau)\,(N\partial_th_N^1)|_{t=t_N\tau}\,d\tau\,=\,\int_0^\infty\chi(\tau)\int_{\T^d}\int_\D\nabla V(x-x_*)\cdot\nabla_vH^{2;t_N\tau}(z,z_*)\,dz_*dx\,d\tau,\]
while the equation for $H^2$ can be solved by Duhamel's formula in the form
\[H^{2;t_N\tau}=\int_0^{t_N\tau} e^{-i(t_N\tau-t')L^\circ}S\,dt'.\]
Using the following formula for Laplace transform (see e.g.~\cite[Lemma~5.2]{DSR-1}), with $g_\chi$ as in the statement,
\[\int_0^\infty\chi(\tau)\int_0^{t_N\tau}e^{-i(t_N\tau-t')L^\circ}dt'\,d\tau\,=\,\int_\R g_\chi(\alpha)\big(iL^\circ+\tfrac{i\alpha+1}{t_N}\big)^{-1}\,d\alpha,\]
the conclusion follows.
\end{proof}

Next, in view of~\eqref{eq:form-resolv-LB}, we provide some basic spectral information on the $2$-particle linearized Vlasov operator $L^\circ$ on $\Ld^1(\D^2)$, cf.~\eqref{eq:LFH0}, which is a densely-defined closable operator on its core $C^\infty_c(\D^2)$,
and we explicitly compute its resolvent.
Henceforth, we use the short-hand notation $\langle f(v)\rangle_v :=\int_{\R^d}f(v)\,dv$ for velocity averages.

\begin{lem}\label{lem:lin-Vl}
Let $V\in W^{1,\infty}(\T^d)$ with $\widehat V\ge0$, let $f^\circ\in W^{1,1}(\R^d)$, and let the operator $iL^\circ$ be split into the Kronecker sum $iL^\circ=iL^\circ_1+iL^\circ_2$ with
\begin{eqnarray*}
iL^\circ_1H&:=&v_1\cdot\nabla_{x_1}H-\nabla_{v_1}f^\circ (v_1)\cdot\int_\D\nabla V(x_1-x_*)H(z_*,z_2)\,dz_*,\\
iL^\circ_2H&:=&v_2\cdot\nabla_{x_2}H-\nabla_{v_2}f^\circ(v_2)\cdot\int_\D\nabla V(x_2-x_*)H(z_1,z_*)\,dz_*.
\end{eqnarray*}
\begin{enumerate}[(i)]
\item The operators $iL^\circ_1$ and $iL^\circ_2$ generate $C_0$-groups $\{e^{itL^\circ_1}\}_{t\in\R}$ and $\{e^{itL^\circ_2}\}_{t\in\R}$ on $\Ld^1(\D^2)$, which commute together, and their sum $iL^\circ$ also generates a $C_0$-group.
\item Further assume that $f^\circ$ is linearly Vlasov-stable in the sense that for any direction~$k$ the projected initial density $\pi_k^\circ(y):=\int_{\R^d}\delta(y-\frac{k\cdot v}{|k|})f^\circ(v)\,dv$ satisfies $y(\pi_k^\circ)'(y)\le0$ for all $y$.
Then, for $j=1,2$, the resolvent of the operator $L_j^\circ$ takes on the following explicit form in Fourier space, for all $\omega\in\C\setminus\R$ and $H\in \Ld^1(\D^2)$,
\begin{gather*}
\qquad(\widehat L^\circ_j-\omega)^{-1}\widehat H\,=\,\frac{\widehat H}{k_j\cdot v_j-\omega}+\frac{\widehat V(k_j)}{\e^\circ(k_j,\omega)}\frac{k_j\cdot\nabla f^\circ(v_j)}{k_j\cdot v_j-\omega}\bigg\langle\frac{\widehat H}{k_j\cdot v_j-\omega}\bigg\rangle_{v_j},\\
\qquad\big\langle(\widehat L^\circ_j-\omega)^{-1}\widehat H\big\rangle_{v_j}\,=\,\frac1{\e^\circ(k_j,\omega)}\bigg\langle\frac{\widehat H}{k_j\cdot v_j-\omega}\bigg\rangle_{v_j},
\end{gather*}
in terms of the dispersion function
\begin{equation}\label{eq:def-eps0}
\qquad\e^\circ(k,\omega)\,:=\,1-\widehat V(k)\bigg\langle\frac{k\cdot\nabla f^\circ(v)}{k\cdot v-\omega}\bigg\rangle_{v},
\end{equation}
which satisfies for all $\omega\in\C\setminus\R$ and $k\in\Z^d$,
\[\qquad|\e^\circ(k,\omega)|\,\ge\,1-\tfrac{|\Re\omega|}{|\omega|}>0.\]
In addition, the resolvent for the sum $L^\circ=L_1^\circ+L_2^\circ$ can be computed via the following integral formula, for all $\Im\omega>0$,
\begin{equation}\label{eq:form-resolv-sum}
\qquad(\widehat L^\circ-\omega)^{-1}=\frac1{2\pi i}\int_\R(\widehat L^\circ_2+\alpha-\tfrac\omega2)^{-1}(\widehat L^\circ_1-\alpha-\tfrac\omega2)^{-1}\,d\alpha.\qedhere
\end{equation}
\end{enumerate}
\end{lem}

\begin{proof}
We start with~(i). The operator $iL_j^\circ$ takes the form $iL_j^\circ=v_j\cdot\nabla_{x_j}+P_j^\circ$, where the transport operator $v_j\cdot\nabla_{x_j}$ generates a $C_0$-group of isometries on $\Ld^1(\D^2)$, and where the perturbation $P_j^\circ$ is bounded in $\Ld^1(\D^2)$, with
\[\|P_j^\circ H\|_{\Ld^1(\D^2)}\le\|\nabla f^\circ\|_{\Ld^1(\R^d)}\|\nabla V\|_{\Ld^\infty(\T^d)}\|H\|_{\Ld^1(\D^2)}.\]
In view of standard perturbation theory~\cite[Theorem~IX.2.1]{Kato}, we deduce that $iL^\circ_j$ itself generates a $C_0$-group on $\Ld^1(\D^2)$, and item~(i) follows.

\medskip\noindent
We turn to~(ii), and we start with the lower bound on the dispersion function $\e^\circ$.
Since~$\widehat V$ is real-valued by assumption, we can compute
\begin{multline*}
|\e^\circ(k,\omega)|^2=\bigg|1-\widehat V(k)\bigg\langle\frac{k\cdot\nabla f^\circ(v)}{k\cdot v-\omega}\bigg\rangle_{v}\bigg|^2\\
=\bigg(1-\widehat V(k)\bigg\langle\frac{(k\cdot v)k\cdot\nabla f^\circ(v)}{(k\cdot v-\Re\omega)^2+(\Im\omega)^2}\bigg\rangle_{v}+(\Re\omega)\widehat V(k)\bigg\langle\frac{k\cdot\nabla f^\circ(v)}{(k\cdot v-\Re\omega)^2+(\Im\omega)^2}\bigg\rangle_{v}\bigg)^2\\
+\bigg((\Im\omega)\widehat V(k)\bigg\langle\frac{k\cdot\nabla f^\circ(v)}{(k\cdot v-\Re\omega)^2+(\Im\omega)^2}\bigg\rangle_{v}\bigg)^2,
\end{multline*}
and we note that the linear Vlasov stability assumption for $f^\circ$ precisely amounts to
\begin{eqnarray*}
1-\widehat V(k)\bigg\langle\frac{(k\cdot v)k\cdot\nabla f^\circ(v)}{(k\cdot v-\Re\omega)^2+(\Im\omega)^2}\bigg\rangle_{v}&=&1-|k|^2\widehat V(k)\int_{\R} \frac{y (\pi_{k}^\circ)'(y)}{(|k|y-\Re\omega)^2+(\Im\omega)^2}dy\\
&\ge&1.
\end{eqnarray*}
In view of the inequality
\[(a+b\Re\omega)^2+(b\Im\omega)^2\ge (a^2+b^2|\omega|^2)(1-\tfrac{|\Re\omega|}{|\omega|}),\qquad\text{for all $a,b\in\R$},\]
we deduce for $\omega\in\C\setminus\R$,
\[|\e^\circ(k,\omega)|^2\ge1-\tfrac{|\Re\omega|}{|\omega|}>0.\]
It remains to compute the resolvent of $L_j^\circ$.
Letting $\omega\in\C\setminus\R$ and $H\in\Ld^1(\D^2)$, we aim to compute $G_\omega:=( L_j^\circ-\omega)^{-1} H$. By definition of $L_j^\circ$, the identity $(L_j^\circ-\omega)G_\omega=H$ takes on the following form in Fourier variables,
\[\widehat H=(k_j\cdot v_j-\omega)\widehat G_\omega-\widehat V(k_j)k_j\cdot\nabla f^\circ(v_j)\langle\widehat G_\omega\rangle_{v_j}\]
or equivalently, dividing by $k_j\cdot v_j-\omega$,
\begin{equation}\label{eq:pre-form-resolv}
\frac{\widehat H}{k_j\cdot v_j-\omega}=\widehat G_\omega-\widehat V(k_j)\frac{k_j\cdot\nabla f^\circ(v_j)}{k_j\cdot v_j-\omega}\langle\widehat G_\omega\rangle_{v_j}.
\end{equation}
Averaging with respect to $v_j$ yields
\[\bigg\langle\frac{\widehat H}{k_j\cdot v_j-\omega}\bigg\rangle_{v_j}=\e^\circ(k_j,\omega)\langle\widehat G_\omega\rangle_{v_j}.\]
with $\e^\circ$ defined in the statement.
As $\e^\circ(\cdot,\omega)$ does not vanish, this yields
\[\langle\widehat G_\omega\rangle_{v_j}=\frac1{\e^\circ(k_j,\omega)}\bigg\langle\frac{\widehat H}{k_j\cdot v_j-\omega}\bigg\rangle_{v_j}.\]
Inserting this into~\eqref{eq:pre-form-resolv}, the explicit formula for the resolvent follows.
Finally, the integral identity~\eqref{eq:form-resolv-sum} for the resolvent of the Kronecker sum $L^\circ=L_1^\circ+L_2^\circ$ is established in~\cite[p.120]{Reed-Simon-73} in case of bounded operators, and is easily extended to the present unbounded setting by an approximation argument as explicit formulas for $(\widehat L_j^\circ-\omega)^{-1}$ ensure that the integral is uniformly convergent (see also~\cite[Lemma~5.5]{DSR-1}).
\end{proof}

We now provide a refined lower bound on the above-defined dispersion function $\e^\circ$, cf.~\eqref{eq:def-eps0}. We further state other related technical estimates that will be crucial to the proof of Corollary~\ref{cor:LB}.

\begin{lem}\label{lem:bound-eps0}
Let $V\in W^{1,\infty}(\T^d)$ and $f^\circ\in \Pc\cap C^\infty_c(\R^d)$.
Assume that $f^\circ$ is linearly Vlasov-stable and that $V$ is positive definite and small enough in the sense of Corollary~\ref{cor:LB}.
Then there holds uniformly for $k\in2\pi\Z^d\setminus\{0\}$ and $\omega\in\C\setminus\R$,
\begin{equation}\label{eq:eps0-finest}
\e^\circ(k,\omega)\gtrsim1,
\end{equation}
and similarly
\begin{equation}\label{eq:apbound-eps+}
|k|\bigg|\bigg\langle \frac{f^\circ(v)}{k\cdot v-\omega}\bigg\rangle_{v}\bigg|\lesssim1,
\end{equation}
where the bounds further depend on $\delta$, $\|V\|_{\Ld^\infty(\T^d)}$, and $\|f^\circ\|_{W^{2+\delta,1}(\R^d)}$ for any $\delta>0$.
In addition, denoting by $\gamma_R$ the contour that is constituted of the real segment $(-\infty,-R]$, the circular arc $\partial B_R\cap\{z\in\C:\Im z<0\}$, and the real segment $[R,\infty)$,
we find
\begin{equation}\label{eq:eps0-infty}
\lim_{R\uparrow\infty}\sup_{k\in2\pi\Z^d}\int_{\gamma_R}|1-\e^\circ(k,\alpha)|\,\frac{d|\alpha|}{|\alpha|}\,=\,0,
\end{equation}
and similarly,
\begin{equation}\label{eq:eps0-infty+}
\lim_{R\uparrow\infty}\sup_{k\in2\pi\Z^d}\int_{\gamma_R}|k|\bigg|\bigg\langle \frac{f^\circ(v)}{k\cdot v-\alpha}\bigg\rangle_{v}\bigg|\,\frac{d|\alpha|}{|\alpha|}\,=\,0.\qedhere
\end{equation}
\end{lem}

\begin{proof}
We first establish~\eqref{eq:eps0-finest} and we start from the lower bound
\begin{equation*}
|\e^\circ(k,\omega)|\,\ge\,|\Re\e^\circ(k,\omega)|
\,\ge\,1-\widehat V(k)\bigg|\bigg\langle\frac{(k\cdot v-\Re\omega)k\cdot\nabla f^\circ(v)}{(k\cdot v-\Re\omega)^2+(\Im\omega)^2}\bigg\rangle_{v}\bigg|.
\end{equation*}
By definition of the projected initial density $\pi_k^\circ(y):=\int_{\R^d}\delta(y-\frac{k\cdot v}{|k|})f^\circ(v)\,dv$, and by symmetry, we can decompose
\begin{multline*}
\bigg|\bigg\langle\frac{(k\cdot v-\Re\omega)k\cdot\nabla f^\circ(v)}{(k\cdot v-\Re\omega)^2+(\Im\omega)^2}\bigg\rangle_{v}\bigg|\,=\,|k|\bigg|\int_\R\frac{(|k|y-\Re\omega)(\pi_k^\circ)'(y)}{(|k|y-\Re\omega)^2+(\Im\omega)^2}dy\bigg|\\
\,\le\,\frac{|k|}2\bigg|\int_{||k|y-\Re\omega|<|k|}\frac{|k|y-\Re\omega}{(|k|y-\Re\omega)^2+(\Im\omega)^2}\Big((\pi_k^\circ)'(y)-(\pi_k^\circ)'\big(2\tfrac{\Re\omega}{|k|}-y\big)\Big)dy\bigg|\\
+|k|\bigg|\int_{||k|y-\Re\omega|>|k|}\frac{(|k|y-\Re\omega)(\pi_k^\circ)'(y)}{(|k|y-\Re\omega)^2+(\Im\omega)^2}dy\bigg|.
\end{multline*}
Using the Hölder continuity of $(\pi_k^\circ)'$ to estimate the first right-hand side term, and using an integration by parts to estimate the second one, we easily deduce for $0<\delta\le1$,
\begin{equation*}
\bigg|\bigg\langle\frac{(k\cdot v-\Re\omega)k\cdot\nabla f^\circ(v)}{(k\cdot v-\Re\omega)^2+(\Im\omega)^2}\bigg\rangle_{v}\bigg|
\,\lesssim_\delta\,\|(\pi_k^\circ)'\|_{W^{\delta,\infty}(\R)}+\|\pi_k^\circ\|_{\Ld^\infty(\R)}.
\end{equation*}
Inserting this bound into the above yields
\begin{equation*}
|\e^\circ(k,\omega)|
\,\ge\,1-C_\delta\|V\|_{\Ld^\infty(\T^d)}\textstyle\sup_k\|\pi_k^\circ\|_{W^{1+\delta,\infty}(\R)}.
\end{equation*}
Using the Sobolev inequality in the form
\[\|\pi_k^\circ\|_{W^{1+\delta,\infty}(\R)}\lesssim\|\pi_k^\circ\|_{W^{2+\delta,1}(\R)}\lesssim\|f^\circ\|_{W^{2+\delta,1}(\R^d)},\]
the claim~\eqref{eq:eps0-finest} follows.

\medskip\noindent
We turn to the proof of~\eqref{eq:apbound-eps+}.
Similarly as above, we decompose
\begin{multline*}
\bigg|\bigg\langle \frac{f^\circ(v)}{k\cdot v-\omega}\bigg\rangle_{v}\bigg|\,=\,\bigg|\int_\R\frac{\pi_k^\circ(y)}{|k|y-\omega}dy\bigg|\\
\,\le\,\frac12\bigg|\int_{||k|y-\Re\omega|<|k|}\bigg(\frac{\pi_k^\circ(y)}{|k|y-\omega}-\frac{\pi_k^\circ(2\frac{\Re\omega}{|k|}-y)}{|k|y-\bar\omega}\bigg)dy\bigg|
+\bigg|\int_{||k|y-\Re\omega|>|k|}\frac{\pi_k^\circ(y)}{|k|y-\omega}dy\bigg|,
\end{multline*}
Writing
\[\frac{\pi_k^\circ(y)}{|k|y-\omega}-\frac{\pi_k^\circ(2\frac{\Re\omega}{|k|}-y)}{|k|y-\bar\omega}=\frac{\pi_k^\circ(y)-\pi_k^\circ(2\frac{\Re\omega}{|k|}-y)}{|k|y-\bar\omega}+\frac{2i(\Im\omega)\pi_k^\circ(y)}{||k|y-\omega|^2},\]
we deduce
\begin{multline*}
\bigg|\bigg\langle \frac{f^\circ(v)}{k\cdot v-\omega}\bigg\rangle_{v}\bigg|
\,\le\,\frac12\bigg|\int_{||k|y-\Re\omega|<|k|}\frac{\pi_k^\circ(y)-\pi_k^\circ(2\frac{\Re\omega}{|k|}-y)}{|k|y-\bar\omega}dy\bigg|\\
+\bigg|\int_{||k|y-\Re\omega|<|k|}\frac{(\Im\omega)\pi_k^\circ(y)}{||k|y-\omega|^2}dy\bigg|+\bigg|\int_{||k|y-\Re\omega|>|k|}\frac{\pi_k^\circ(y)}{|k|y-\omega}dy\bigg|.
\end{multline*}
This entails for $0<\delta\le1$,
\begin{equation*}
\bigg|\bigg\langle \frac{f^\circ(v)}{k\cdot v-\omega}\bigg\rangle_{v}\bigg|\,\lesssim_\delta\,|k|^{-1}\|\pi_k^\circ\|_{\Ld^1\cap W^{\delta,\infty}(\R)}\,\lesssim\,|k|^{-1}\|f^\circ\|_{W^{1+\delta,1}(\R^d)},
\end{equation*}
and the claim~\eqref{eq:apbound-eps+} follows.

\medskip\noindent
We turn to the proof of~\eqref{eq:eps0-infty}. Let $k\in2\pi\Z^d\setminus\{0\}$ be fixed. By definition~\eqref{eq:def-eps0} of $\e^\circ$, we can decompose as above,
\begin{multline*}
|1-\e^\circ(k,\omega)|\,=\,\widehat V(k)\bigg|\bigg\langle\frac{k\cdot\nabla f^\circ(v)}{k\cdot v-\omega}\bigg\rangle_v\bigg|\,=\,|k|\widehat V(k)\bigg|\int_\R\frac{(\pi_k^\circ)'(y)}{|k|y-\omega}dy\bigg|\\
\,\le\,\frac{|k|}2\widehat V(k)\bigg|\int_{||k|y-\Re\omega|<|k|}\frac{(\pi_k^\circ)'(y)-(\pi_k^\circ)'(2\frac{\Re\omega}{|k|}-y)}{|k|y-\bar\omega}dy\bigg|\\
+|k|\widehat V(k)\bigg|\int_{||k|y-\Re\omega|<|k|}\frac{(\Im\omega)(\pi_k^\circ)'(y)}{||k|y-\omega|^2}dy\bigg|+|k|\widehat V(k)\bigg|\int_{||k|y-\Re\omega|>|k|}\frac{(\pi_k^\circ)'(y)}{|k|y-\omega}dy\bigg|.
\end{multline*}
Using the Hölder continuity of $(\pi_k^\circ)'$ to estimate the first right-hand side term, and using integration by parts to estimate the last term, distinguishing between the contribution of $||k|y-\Re\omega|>(\frac12|\Re\omega|^\delta)\vee|k|$ and of $|k|<||k|y-\Re\omega|<(\frac12|\Re\omega|^\delta)\vee|k|$, we easily deduce for~$0<\delta\le1$,
\begin{multline*}
|1-\e^\circ(k,\omega)|
\,\lesssim\,\widehat V(k)\|(\pi_k^\circ)'\|_{W^{\delta,\infty}(\{y:||k|y-\Re\omega|<|k|\})}\\
+\widehat V(k)\|\pi_k^\circ\|_{\Ld^1(\{y:||k|y-\Re\omega|<(\frac12|\Re\omega|^\delta)\vee|k|\})}
+|k|\big(|\Re\omega|^\delta+|\Im\omega|+|k|\big)^{-1}\widehat V(k) \|\pi_k^\circ\|_{\Ld^\infty(\R)},
\end{multline*}
hence, by the Sobolev inequality,
\begin{multline*}
|1-\e^\circ(k,\omega)|
\,\lesssim\,\widehat V(k)\|\pi_k^\circ\|_{W^{2+\delta,1}(\{y:||k|y-\Re\omega|<|k|\})}\\
+\widehat V(k)\|\pi_k^\circ\|_{\Ld^1(\{y:||k|y-\Re\omega|<(\frac12|\Re\omega|^\delta)\vee|k|\})}+|k|\big(|\Re\omega|^\delta+|\Im\omega|+|k|\big)^{-1}\widehat V(k) \|\pi_k^\circ\|_{\Ld^\infty(\R)}.
\end{multline*}
Integrating over the contour $\gamma_R$ and using Fubini's theorem, we find
\begin{equation*}
\int_{\gamma_R}|1-\e^\circ(k,\alpha)|\frac{d|\alpha|}{|\alpha|}
\,\lesssim_\delta\,R^{-1}|k|\widehat V(k)\big(\|\pi_k^\circ\|_{W^{2+\delta,1}(\R)}+\|\langle\cdot\rangle^\delta\pi_k^\circ\|_{\Ld^1(\R)}\big)+R^{-\delta}\|\pi_k^\circ\|_{\Ld^\infty(\R)},
\end{equation*}
and the conclusion~\eqref{eq:eps0-infty} follows.
Finally, the proof of~\eqref{eq:eps0-infty+} is similar and is omitted.
\end{proof}

With the above estimates at hand, Corollary~\ref{cor:LB} is easily deduced by taking inspiration from the formal computations in Nicholson's textbook~\cite[Appendix~A]{Nicholson}, similarly as we have done in~\cite{DSR-1}.
More precisely, Corollary~\ref{cor:LB} is a direct consequence of the following computation together with Lemma~\ref{lem:laplace} and Lebesgue's dominated convergence theorem.

\begin{lem}
Let $V\in W^{1,\infty}(\T^d)$ and $f^\circ\in \Pc\cap C^\infty_c(\R^d)$.
Assume that $f^\circ$ is linearly Vlasov-stable and that $V$ is positive definite and small enough in the sense of Corollary~\ref{cor:LB}.
Then, with the notation~\eqref{eq:H-defin0} for~$S$, there holds
\begin{equation*}
\lim_{\omega\to0\atop\Im\omega>0,\,|\Re\omega|\lesssim\Im\omega}\int_{\T^d}\int_\D\nabla V(x-x_*)\cdot\Big(\nabla_v\big(iL^\circ-i\omega\big)^{-1}S\Big)(z,z_*)\,dz_*dx\,=\,\LB(f^\circ),
\end{equation*}
where we recall that the Lenard--Balescu operator $\LB$ is defined in~\eqref{eq:LB}.
In addition, the argument of the limit can be written as $\nabla_v\cdot T_\omega(v)$, where $T_\omega$ satisfies for all $v\in\R^d$ and $\omega\in\C\setminus\R$,
\begin{equation}\label{eq:bound-T}
|T_\omega(v)|\,\lesssim\,(|f^\circ(v)|+|\nabla f^\circ(v)|)\log\Big(2+\frac{|\Re\omega|}{|\Im\omega|}\Big),
\end{equation}
where the multiplicative constant further depends on $\delta$, $\|V\|_{W^{1,\infty}(\T^d)}$, and $\|f^\circ\|_{W^{2+\delta,1}(\R^d)}$ for any $\delta>0$.
\end{lem}

\begin{proof}
Given $\Im\omega>0$,
we can write in Fourier variables, using~\eqref{eq:form-resolv-sum} to express the resolvent of the $2$-particle linearized Vlasov operator $L^\circ$,
\begin{eqnarray*}
\lefteqn{\int_{\T^d}\int_\D\nabla V(x-x_*)\cdot\Big(\nabla_v\big(iL^\circ-i\omega\big)^{-1}S\Big)(z,z_*)\,dz_*dx}\\
&=&\nabla_v\cdot\sum_{k\in2\pi\Z^d}k\widehat V(k)\Big\langle \big((\widehat L^\circ-\omega)^{-1}\widehat S\big)(-k,v,k,v_*)\Big\rangle_{v_*}\\
&=&\nabla_v\cdot\sum_{k\in2\pi\Z^d}k\widehat V(k)\frac1{2\pi i}\int_\R\Big\langle \big((\widehat L_2^\circ+\alpha-\tfrac\omega2)^{-1}(\widehat L_1^\circ-\alpha-\tfrac\omega2)^{-1}\widehat S\big)(-k,v,k,v_*)\Big\rangle_{v_*}d\alpha.
\end{eqnarray*}
Noting that
\[\widehat S(-k,v,k,v_*):=-ik\widehat V(k)\cdot(\nabla_{v}-\nabla_{v_*})(f^\circ\otimes f^\circ)(v,v_*),\]
and inserting the explicit computation of the resolvents of $L_1^\circ$ and $L_2^\circ$ as obtained in Lemma~\ref{lem:lin-Vl}(ii), we find
\begin{multline*}
\int_{\T^d}\int_\D\nabla V(x-x_*)\cdot\Big(\nabla_v\big(iL^\circ-i\omega\big)^{-1}S\Big)(z,z_*)\,dz_*dx\\
=\nabla_v\cdot\sum_{k\in2\pi\Z^d}k\widehat V(k)\frac1{2\pi}\int_\R\frac1{\e^\circ(k,\frac\omega2-\alpha)}\frac1{k\cdot v+\alpha+\tfrac\omega2}\\
\times\Bigg( \widehat V(k)\bigg\langle\frac{k\cdot(\nabla_{v}-\nabla_{v_*})f^\circ(v) f^\circ(v_*)}{k\cdot v_*+\alpha-\frac\omega2}\bigg\rangle_{v_*}\\
+\widehat V(k)^2\frac{k\cdot\nabla f^\circ(v)}{\e^\circ(-k,\tfrac\omega2+\alpha)}
\bigg\langle \frac{k\cdot(\nabla_{w_*}-\nabla_{v_*})f^\circ(w_*)f^\circ(v_*)}{(k\cdot v_*+\alpha-\frac\omega2)(k\cdot w_*+\alpha+\tfrac\omega2)}\bigg\rangle_{v_*,w_*}\Bigg)\,d\alpha.
\end{multline*}
By definition~\eqref{eq:def-eps0} of the dispersion function $\e^\circ$, we can write
\begin{equation}\label{eq:rewr-eps0}
\widehat V(k)\bigg\langle \frac{k\cdot\nabla f^\circ(v_*)}{k\cdot v_*+\alpha-\frac\omega2}\bigg\rangle_{v_*}=1-\e^\circ(k,\tfrac\omega2-\alpha),
\end{equation}
and similarly,
\begin{multline*}
\widehat V(k)\bigg\langle \frac{k\cdot(\nabla_{w_*}-\nabla_{v_*})f^\circ(w_*)f^\circ(v_*)}{(k\cdot v_*+\alpha-\frac\omega2)(k\cdot w_*+\alpha+\tfrac\omega2)}\bigg\rangle_{v_*,w_*}\\
=\bigg\langle \frac{f^\circ(v_*)}{k\cdot v_*+\alpha-\frac\omega2}\bigg\rangle_{v_*}\big(1-\e^\circ(-k,\tfrac\omega2+\alpha)\big)
-\bigg\langle \frac{f^\circ(w_*)}{k\cdot w_*+\alpha+\tfrac\omega2}\bigg\rangle_{w_*}\big(1-\e^\circ(k,\tfrac\omega2-\alpha)\big).
\end{multline*}
Inserting these identities into the above, we find after simplifications,
\begin{multline*}
\int_{\T^d}\int_\D\nabla V(x-x_*)\cdot\Big(\nabla_v\big(iL^\circ-i\omega\big)^{-1}S\Big)(z,z_*)\,dz_*dx\\
=\nabla_v\cdot\sum_{k\in2\pi\Z^d}k\widehat V(k)\frac1{2\pi}\int_\R\frac1{k\cdot v+\alpha+\tfrac\omega2}\\
\times\Bigg( \bigg(1-\frac1{\e^\circ(k,\frac\omega2-\alpha)}\bigg)\bigg(f^\circ(v)+\widehat V(k)\frac{k\cdot\nabla f^\circ(v)}{\e^\circ(-k,\tfrac\omega2+\alpha)}\bigg\langle \frac{f^\circ(w_*)}{k\cdot w_*+\alpha+\tfrac\omega2}\bigg\rangle_{w_*}\bigg)\\
+\widehat V(k)\frac{k\cdot\nabla f^\circ(v)}{\e^\circ(k,\frac\omega2-\alpha)\e^\circ(-k,\tfrac\omega2+\alpha)}\bigg\langle \frac{f^\circ(v_*)}{k\cdot v_*+\alpha-\frac\omega2}\bigg\rangle_{v_*}\Bigg)\,d\alpha.
\end{multline*}
As $\Im\omega>0$, we note that the integrand
\[\alpha\mapsto1-\frac1{\e^\circ(k,\frac\omega2-\alpha)}\]
is holomorphic on the lower complex half-plane $\Im\alpha<\frac12\Im\omega$.
In view of~\eqref{eq:eps0-finest} and~\eqref{eq:eps0-infty} in Lemma~\ref{lem:bound-eps0},
complex deformation then yields
\[\frac1{2\pi i}\int_\R\frac1{k\cdot v+\alpha+\frac\omega2}\bigg(1-\frac1{\e^\circ(k,\frac\omega2-\alpha)}\bigg)\,d\alpha\,=\,\frac1{\e^\circ(k,k\cdot v+\omega)}-1,\]
which is the residue at $\alpha=-k\cdot v-\frac\omega2$.
Similarly, we compute
\[\int_\R\frac1{k\cdot v+\alpha+\frac\omega2}\frac{1}{\e^\circ(-k,\tfrac\omega2+\alpha)}\bigg\langle \frac{f^\circ(w_*)}{k\cdot w_*+\alpha+\tfrac\omega2}\bigg\rangle_{w_*}d\alpha=0.\]
Inserting these identities into the above yields
\begin{multline}\label{eq:rewr-LB-r}
\int_{\T^d}\int_\D\nabla V(x-x_*)\cdot\Big(\nabla_v\big(iL^\circ-i\omega\big)^{-1}S\Big)(z,z_*)\,dz_*dx\\
=\nabla_v\cdot\sum_{k\in2\pi\Z^d}ik\widehat V(k) f^\circ(v)\bigg(\frac1{\e^\circ(k,k\cdot v+\omega)}-1\bigg)
+\nabla_v\cdot\sum_{k\in2\pi\Z^d}k\widehat V(k)L_\omega(k,v),
\end{multline}
in terms of
\begin{multline*}
L_\omega(k,v)\,:=\,\frac1{2\pi}\int_\R\frac1{k\cdot v+\alpha+\tfrac\omega2}\\
\times\frac{\widehat V(k)k\cdot\nabla f^\circ(v)}{\e^\circ(k,\frac\omega2-\alpha)\e^\circ(-k,\tfrac\omega2+\alpha)}\bigg(\bigg\langle \frac{f^\circ(v_*)}{k\cdot v_*+\alpha-\frac\omega2}\bigg\rangle_{v_*}-\bigg\langle \frac{f^\circ(v_*)}{k\cdot v_*+\alpha+\tfrac\omega2}\bigg\rangle_{v_*}\bigg)\,d\alpha.
\end{multline*}
Since $V$ is even, its Fourier transform $\widehat V$ is even too, hence we find by symmetry,
\begin{equation}\label{eq:LB-even-re}
\sum_{k\in2\pi\Z^d}k\widehat V(k)L_\omega(k,v)=\sum_{k\in2\pi\Z^d}k\widehat V(k)\widetilde L_\omega(k,v),
\end{equation}
in terms of
\begin{multline*}
\widetilde L_\omega(k,v)\,:=\,\frac12\big(L_\omega(k,v)-L_\omega(-k,v)\big)
\,=\,\frac1{2\pi}\int_\R\frac12\bigg(\frac1{k\cdot v+\alpha+\tfrac\omega2}-\frac1{k\cdot v+\alpha-\tfrac\omega2}\bigg)\\
\times\frac{\widehat V(k)k\cdot\nabla f^\circ(v)}{\e^\circ(k,\frac\omega2-\alpha)\e^\circ(-k,\tfrac\omega2+\alpha)}\bigg(\bigg\langle \frac{f^\circ(v_*)}{k\cdot v_*+\alpha-\frac\omega2}\bigg\rangle_{v_*}-\bigg\langle \frac{f^\circ(v_*)}{k\cdot v_*+\alpha+\tfrac\omega2}\bigg\rangle_{v_*}\bigg)\,d\alpha.
\end{multline*}
In view of Lemma~\ref{lem:bound-eps0}, it is easily checked that
\begin{equation*}
|\widetilde L_\omega(k,v)|\,\lesssim\,\widehat V(k) |\nabla f^\circ(v)|\int_\R\frac{|\omega|}{|\alpha^2-\omega^2|}\,d\alpha\,\lesssim\,\widehat V(k) |\nabla f^\circ(v)|\log\Big(2+\frac{\Re\omega}{\Im\omega}\Big).
\end{equation*}
Combining this with~\eqref{eq:rewr-LB-r} and~\eqref{eq:LB-even-re}, we can write as stated
\[\int_{\T^d}\int_\D\nabla V(x-x_*)\cdot\Big(\nabla_v\big(iL^\circ-i\omega\big)^{-1}S\Big)(z,z_*)\,dz_*dx\,=\,\nabla_v\cdot T_\omega(v),\]
with $T_\omega$ satisfying the bound~\eqref{eq:bound-T}.
It remains to pass to the limit $\omega\to0$ with $\Im\omega>0$ and $|\Re\omega|\lesssim\Im\omega$.
Recall the Sokhotskii-Plemelj formula
\begin{equation}\label{eq:Plemelj}
\Im\frac1{k\cdot v+\alpha+i0}=-\pi\delta(k\cdot v+\alpha),
\end{equation}
which we use in the following form, for all test functions $\phi\in C_b(\R)$,
\[\lim_{\omega\to0\atop\Im\omega>0,\,|\Re\omega|\lesssim\Im\omega}\int_\R\frac12\bigg(\frac1{k\cdot v+\alpha+\tfrac\omega2}-\frac1{k\cdot v+\alpha-\tfrac\omega2}\bigg)\phi(\alpha)\,d\alpha\,=\,-i\pi\phi(-k\cdot v).\]
Further noting that $\e^\circ(-k,\alpha+i0)=\overline{\e^\circ(k,-\alpha+i0)}$, we deduce
\begin{equation*}
\widetilde L(k,v):=\lim_{\omega\to0\atop\Im\omega>0,\,|\Re\omega|\lesssim\Im\omega}\widetilde L_\omega(k,v)
=\frac{\widehat V(k)k\cdot\nabla f^\circ(v)}{|\e^\circ(k,k\cdot v+i0)|^2}\bigg\langle\Im \frac{f^\circ(v_*)}{k\cdot (v_*-v)-i0}\bigg\rangle_{v_*}\,d\alpha.
\end{equation*}
Using the Sokhotskii-Plemelj formula~\eqref{eq:Plemelj} again,
we find
\begin{equation*}
\widetilde L(k,v)
=\frac{\pi\widehat V(k)k\cdot\nabla f^\circ(v)}{|\e^\circ(k,k\cdot v+i0)|^2}\big\langle f^\circ(v_*)\delta(k\cdot (v_*-v))\big\rangle_{v_*}.
\end{equation*}
Similarly, we compute
\[\Im\bigg(\frac1{\e^\circ(k,k\cdot v+i0)}-1\bigg)=\frac{\Im\e^\circ(k,k\cdot v-i0)}{|\e^\circ(k,k\cdot v+i0)|^2}=\frac{\pi\widehat V(k)\langle k\cdot\nabla f^\circ(v_*)\delta(k\cdot(v_*-v))\rangle_{v_*}}{|\e^\circ(k,k\cdot v+i0)|^2}.\]
Combining these computations with~\eqref{eq:rewr-LB-r} and~\eqref{eq:LB-even-re}, and noting that $\e^\circ(k,k\cdot v-i0)=\e(k,k\cdot v;\nabla f^\circ)$ (comparing definitions~\eqref{eq:LB-disp} and~\eqref{eq:def-eps0}), the conclusion follows.
\end{proof}

\section*{Acknowledgements}
The author wishes to warmly thank François Golse, Laure Saint-Raymond, and Sergio Simonella for motivating discussions.
His work is supported by the CNRS-Momentum program.

\bibliographystyle{plain}
\bibliography{biblio}

\end{document}